%% file: intfrac_hp.tex
\RequirePackage{fix-cm}
\documentclass[envcountsame,envcountsect]{svjour3}

\usepackage[dvipsnames]{xcolor}
\usepackage{amsfonts,amsmath,amssymb}
\usepackage{mathrsfs,mathtools}
\usepackage{subcaption}

\smartqed
\usepackage{tgpagella}
\usepackage{aligned-overset}
\usepackage{enumitem}
\usepackage{xspace}
\usepackage[hidelinks]{hyperref}
\usepackage{esint}
\usepackage{graphicx}
\usepackage{tikz}
\usepackage{psfrag}
\usepackage[percent]{overpic}
\usepackage[normalem]{ulem}
\usepackage{color}
\usepackage[most]{tcolorbox}

\input{preamble.tex}

\begin{document}
\title{Exponential Convergence of $hp$ FEM for the
\\
Integral Fractional Laplacian in Polygons
\thanks{
The research of JMM was supported by the Austrian Science Fund (FWF) project F 65. 
}
}
\titlerunning{Exponential Convergence for Integral Fractional Laplacian}  

\author{
Markus Faustmann \and
Carlo Marcati \and
\\
Jens M. Melenk \and
Christoph Schwab}

\institute{
M. Faustmann \at Institut  f\"{u}r  Analysis und  Scientific Computing\\ 
Technische Universit\"{a}t Wien\\ A-1040  Vienna, Austria \\ \email{markus.faustmann@tuwien.ac.at}
\and
C. Marcati \at Dipartimento di Matematica\\
Universit{\`a} di Pavia\\
I-27100 Pavia, Italy\\
\email{carlo.marcati@unipv.it}
\and
J.M. Melenk \at Institut  f\"{u}r  Analysis und  Scientific Computing\\ 
Technische Universit\"{a}t Wien\\ A-1040  Vienna, Austria \\\email{melenk@tuwien.ac.at}
\and
Ch. Schwab \at Seminar for Applied Mathematics\\ ETH Z\"{u}rich, ETH Zentrum, 
HG  G57.1\\ CH8092 Z\"{u}rich, Switzerland \\ \email{christoph.schwab@sam.math.ethz.ch}
}

\maketitle
\begin{abstract}
We prove exponential convergence in the energy norm of $hp$ finite element discretizations 
for the integral fractional diffusion operator of order $2s\in (0,2)$
subject to homogeneous Dirichlet boundary conditions 
in bounded polygonal domains $\Omega\subset \bbR^2$.
Key ingredient in the analysis are  
the weighted analytic regularity from \cite{FMMS21_983} and 
meshes that feature anisotropic geometric refinement towards $\partial\Omega$.
\end{abstract}
\section{Introduction}
\label{S:introduction}
%
In recent years, mathematical and computational modelling 
in engineering and natural sciences has witnessed the emergence of 
nonlocal boundary value problems and their mathematical and numerical analysis.
For applications of fractional models, 
we refer to the surveys \cite{GunzbActa,BBNOS18,RosOton2016Surv,AinsworthEtAl_FracSurv2018} 
and the references therein.

A typical nonlocal, elliptic equation is the so-called fractional Laplacian.
In a bounded domain $\Omega\subset \bbR^d$, and for $s\in (0,1)$,
the Dirichlet problem of the fractional Laplacian reads, informally, 
for given $f:\Omega \to \bbR$, to find $u:\bbR^d\to \bbR$ such that
\begin{equation}\label{eq:FracLap}
(-\Delta)^s u = f \quad \mbox{in}\quad \Omega, 
\qquad 
u = 0 \quad \mbox{in} \;\;\Omega^c := \bbR^d\backslash \overline{\Omega}
.
\end{equation}
Nonlocality manifests here in that
the operator $(-\Delta)^s$ acts on $u$ globally (see \eqref{eq:IntFrc} ahead), 
and that 
the Dirichlet ``boundary'' condition is, in fact,
a condition on the unknown on the whole exterior of $\Omega$.
\subsection{Integral Fractional Diffusion}
\label{sec:FracDiff}
%
%
We consider a bounded, open polygon
$\Omega \subset \bbR^2$
with Lipschitz boundary $\partial \Omega$ consisting
of a finite number of straight sides (the edges of the polygon) and vertices.
For $s\in (0,1)$, there are various different possible definitions of the fractional Laplacian $(-\Delta)^s$ (cf. \cite{Kwasnicki}), 
which are equivalent on the full-space, 
but may differ on bounded domains. 
Here, we study the 
integral (Dirichlet) fractional Laplacian $(-\Delta)^s$ that,
acting on a sufficiently regular function $u$ in $\Omega$, 
reads
\begin{equation}\label{eq:IntFrc}
(-\Delta)^s u(x) := C(s, d) \mbox{P.V.} \int_{\bbR^2} \frac{u(x)-u(z)}{|x-z|^{2+2s}} dz 
\;,
\quad 
C(s,d)\coloneqq - 2^{2s}\frac{\Gamma(s+d/2)}{\pi^{d/2}\Gamma(-s)}.
\end{equation}
Here, $\mbox{P.V.}$ denotes the Cauchy principal value integral.

In order to state a variational formulation of 
\eqref{eq:FracLap}, fractional order Sobolev spaces are required.
For integer order $t \in \bbN_0$ and domain $\omega\subset \R^d$,  we denote by $H^t(\omega)$ the Hilbertian Sobolev spaces.
Fractional order Sobolev spaces for $t \in (0,1)$ are
defined through the Slobodeckij seminorm $|\cdot|_{H^t(\omega)}$,
and the corresponding norm $\|\cdot\|_{H^t(\omega)}$,
given by
\begin{align}\label{eq:FracNrm}
|v|^2_{H^t(\omega)}
=
\int_{\omega} \int_{\omega} \frac{|v(x) - v(z)|^2}{\abs{x-z}^{2+2t}}
\,dz\,dx,
\;\;
\|v\|^2_{H^t(\omega)} = \|v\|^2_{L^2(\omega)} + |v|^2_{H^t(\omega)}.
\end{align}
For $t \in (0,1)$, we employ the spaces
\begin{align} \label{eq:Htilde}
\widetilde{H}^{t}(\Omega) 
:= 
\left\{u \in H^t(\R^d) : u\equiv 0 \; \text{on} \; \R^d \backslash \overline{\Omega} \right\},
\;  
\norm{v}_{\widetilde{H}^{t}(\Omega)}^2 
:= 
\norm{v}_{H^t(\Omega)}^2 + \norm{v/r^t}_{L^2(\Omega)}^2.
\end{align}
Here and throughout, $r(x):=\operatorname{dist}(x,\partial\Omega)$
denotes the Euclidean distance of a point $x \in \Omega$ from the boundary $\partial \Omega$.
For $t > 0$, the space $H^{-t}(\Omega)$ denotes the dual space of $\widetilde{H}^t(\Omega)$,
and $\skp{\cdot,\cdot}_{L^2(\Omega)}$ denotes 
the duality pairing that extends the $L^2(\Omega)$-inner product.

The variational form of \eqref{eq:FracLap} 
reads: find $u \in \widetilde{H}^s(\Omega)$ such that, 
for all $v \in \widetilde{H}^s(\Omega)$,
\begin{equation}
\label{eq:weakform}
a(u,v):= \frac{C(s, d)}{2} \int_{\R^2}\int_{\R^2}
 \frac{(u(x)-u(z))(v(x)-v(z))}{\abs{x-z}^{2+2s}} \, dz \, dx = \skp{f,v}_{L^2(\Omega)}
\end{equation}
Existence and uniqueness of $u \in \widetilde{H}^s(\Omega)$ follow from
the Lax--Milgram Lemma for any $f \in H^{-s}(\Omega)$, upon the observation
that the bilinear form 
$a(\cdot,\cdot): \widetilde{H}^s(\Omega)\times \widetilde{H}^s(\Omega)\to \R$
is continuous and coercive,
see, e.g., \cite[Sec.~{2.1}]{acosta-borthagaray17}.

This observation implies that for any subspace $V_N\subset \widetilde{H}^s(\Omega)$ 
of finite dimension $N$,
the Galerkin discretization:
\begin{equation}\label{eq:GalV}
u_N\in V_N : \quad a(u_N,v) = \skp{f,v}_{L^2(\Omega)} \quad \forall v\in V_N
\end{equation}
admits a unique solution $u_N\in V_N$. Whence 
\begin{equation}\label{eq:QuasiOpt}
\forall v_N \in V_N:\quad 
\| u - u_N \|_{\widetilde{H}^s(\Omega)} 
\leq C 
\| u - v_N \|_{\widetilde{H}^s(\Omega)} 
\;.
\end{equation}
Convergence rates depend on the regularity of $u$ and on the structure of 
$\{ V_N \}_{N\in \bbN}$. 
We establish \emph{exponential convergence rate bounds}
for the right hand side of \eqref{eq:QuasiOpt}
under \emph{weighted, analytic regularity} of $u$
in vertex- and edge-weighted spaces in $\Omega$.
This requires $\{ V_N \}_{N\in \bbN}$ to be a family of 
finite-dimensional subspaces of $hp$-type. 
In particular, we construct a family $\{ \Pi_N \}_{N\in \N}$
of spectral element approximation operators such that 
exponential convergence rate bounds are attained in \eqref{eq:QuasiOpt} with
$v_N = \Pi_N u_N$.
\subsection{Existing Results}
\label{sec:ExRes}
In the recent work \cite{BN21}, the regularity of the solution $u$ 
of \eqref{eq:FracLap} in a certain (isotropic) Besov space on 
Lipschitz domains $\Omega$ was shown. 
This was subsequently used in 
\cite{borthagaray2021constructive} to infer algebraic 
convergence rates of Galerkin FEM in \eqref{eq:QuasiOpt},
where, in \cite{borthagaray2021constructive}, 
the spaces $\{ V_N \}_{N\in \bbN}$ are
a family of continuous, piecewise affine
Lagrangian first order Finite Elements, 
on a sequence of shape-regular triangulations
in $\Omega$ with judicious, \emph{isotropic boundary refinement}.
The necessity of such refinement
can be expected by the boundary asymptotics of the solution 
shown, e.g., in \cite{ros-oton-serra14}, 
where $u(x) \sim \dist(x,\partial\Omega)^s$ was established.

The anisotropic nature of the edge-singularities of the solution $u$
in $\Omega$ precludes high convergence rates (in terms of 
error versus number of degrees of freedom) 
for FE discretizations based on shape-regular mesh families: 
\emph{anisotropic boundary refinement} is necessary to this end.

The regularity of solutions to \eqref{eq:FracLap} has been studied
intensively in recent years.
\cite{ros-oton-serra14,Abels2020} 
established H\"older regularity of solutions in $\overline{\Omega}$, 
when $\partial\Omega$ is $C^1$, with asymptotic behavior as 
$\dist(x,\partial\Omega)^s$ for $x\in \Omega$ 
(corner domains $\Omega\subset \bbR^2$ as considered here
 are not covered by these results).
In \cite{GiEPSSt,stocek} 
vertex- and edge-singularities of solutions 
to \eqref{eq:FracLap} have been investigated formally, and 
the dominant singular terms of weak solutions $u \in \widetilde{H}^s(\Omega)$
of \eqref{eq:weakform} have been calculated, 
under provision of 
sufficiently high (finite) regularity of $f$ in \eqref{eq:FracLap}.

In \cite{FMMS21_983}, we studied elliptic regularity for \eqref{eq:FracLap}
in the case that 
a) $\Omega\subset \bbR^2$ is a polygon, with (a finite number of) straight sides,
and 
b) the data $f$ in \eqref{eq:FracLap} is analytic in $\overline{\Omega}$.
We detail the results of \cite{FMMS21_983} in Section~\ref{sec:Reg} ahead;
they constitute the basis of the proof of the main result of the present
paper, the exponential convergence rate bound \eqref{eq:hpExpConv}.

There are are several constructions of fractional
powers of the (Dirichlet) Laplacian. 
Besides the integral fractional Laplacian considered here 
and, e.g., in \cite{ros-oton-serra14,BN21,GiEPSSt}, we mention
the related, so-called \emph{spectral fractional Laplacian},
for which regularity and FE analysis was considered,
e.g., in \cite{NOS,AG17,BMNOSS17_732,BMS20_2880}.
In
\cite{BMNOSS17_732,BMS20_2880},
exponential convergence of $hp$-FEM for the spectral fractional diffusion 
in 
so-called curvilinear polygonal domains, subject to analytic data,
was proved. 
The mathematical analysis and the numerical method in these references
leveraged the reformulation of the nonlocal boundary value problem 
in terms of a degenerate, elliptic \emph{local} boundary value problem,
which can be approximated by a collection of (still local) elliptic
singular perturbation problems, for which $hp$-FEM have been shown
to deliver exponential convergence rates in \cite{melenk02,banjai-melenk-schwab19-RD}.

Numerical analysis for the integral fractional Laplacian was developed
also in recent contributions 
\cite{BLP19,acosta-borthagaray17,faustmann-karkulik-melenk20,KarMel19}.
We refer to the surveys 
\cite{BBNOS18,AinsworthEtAl_FracSurv2018} and the references there
for a comprehensive presentation and references.
None of these references establishes, in space dimension $d>1$, 
exponential rates of convergence.
\subsection{Contributions}
\label{S:contrib}
We prove exponential rate of convergence
of solutions for the homogeneous Dirichlet problem 
of the integral fractional Laplacian of order $2s\in (0,2)$ 
in polygonal domains $\Omega\subset \bbR^2$, 
subject to a source term $f$ that is analytic in $\overline{\Omega}$.

We resolve the vertex- and edge-singularities, which are 
well-known to occur due to the singular support of the solution $u$ being all of $\partial \Omega$ 
(see, e.g., \cite{Grubb15,Abels2020,GiEPSSt}) 
by 
\emph{anisotropic, geometric mesh refinement towards $\partial \Omega$}.
The class of admissible geometric meshes in $\Omega$ will consist of 
a finite union of patchwise structured geometric partitions that are images 
of partitions from a finite catalog ${\mathfrak P}$, 
as depicted in Fig.~\ref{fig:patches} below, 
similar to the construction in \cite{banjai-melenk-schwab19-RD,BMS20_2880}.
The structured, anisotropic geometric partitions
in the patches
are assumed to be obtained by a finite number $L$ of bisections.
On the corresponding global geometric partition in $\Omega$, 
the $hp$-approximation space $V_N$ in \eqref{eq:GalV}, \eqref{eq:QuasiOpt}
consists of continuous, piecewise polynomials of degree $q \sim L \geq 1$.

The principal result of the present paper can be stated as follows.
\begin{theorem}\label{thm:hpExpConv}
Let $\Omega \subset \bbR^2$ be a polygon. 
There is a sequence $\{ V_N \}_{N\geq 1}$ of $hp$-Finite Element spaces, 
with dimension not exceeding $N$, 
such that for $f$ that is analytic in $\overline{\Omega}$ and the 
solution $u$ of (\ref{eq:weakform}), the Galerkin approximations $u_N \in V_N$ 
of \eqref{eq:GalV} converge exponentially to $u$, i.e., there 
are constants $b$, $C >0$ (depending on $s$, $\Omega$, and $f$) 
such that
\begin{equation}\label{eq:hpExpConv}
\| u - u_N \|_{\tH^s(\Omega)} \leq C \exp(-b \sqrt[4]{N}).
\end{equation}
The spaces $V_N$ can be taken as the spaces $W^L_q$ (see \eqref{eq:S^q_0} for the precise definition), 
which are spaces of globally continuous, 
piecewise mapped polynomials of degree $q$ on boundary-refined meshes $\Tg$ 
(see Def.~\ref{def:bdylayer-mesh}) with $L$ layers of geometric refinement, 
for $L\sim q \sim N^{1/4}$.
\end{theorem}
%
\subsection{Layout}
\label{S:outline}
In Section~\ref{sec:Reg}, we recapitulate the weighted, analytic regularity
results of~\cite{FMMS21_983}, which form the basis of the proofs of 
the exponential convergence. In Section~\ref{sec:LoclNrm}, we state an embedding
result of weighted, integer order spaces $H^1_\beta(\Omega)$ into fractional
ones, which will be instrumental in the following analysis as local constructions can easily be done in those spaces. 
Section~\ref{sec:mesh} contains the definition of the $hp$-FE spaces,
in particular of the structured geometric meshes on the reference patches,
which are simplifications of the constructions used in
\cite{banjai-melenk-schwab19-RD,BMS20_2880}.
Section~\ref{sec:hp-approx} has the key exponential approximation error bounds 
in the weighted, \emph{local} $H^1_\beta(\Omega)$-norm for the $hp$-FE spaces on the
geometric, boundary-refined meshes in the patches.
This is followed by the proof of Theorem~\ref{thm:hpExpConv}.

Appendix~\ref{sec:approx-reference-element} recapitulates the Gauss-Lobatto
interpolants in the reference elements together with their basic approximation
and stability properties from \cite{melenk02,banjai-melenk-schwab19-RD}.
In Appendix \ref{sec:norm-cutoff}, we show some technical lemmas used in the
proof of the main result.
\subsection{Notation} 
\label{sec:Notat}
Constants $C$ may be different in each occurence, 
but are independent of critical parameters of the discretization such as $N,p,L$.
We denote by $\widehat{S}\coloneqq (0,1)^2$ the reference square and by 
$\widehat{T} \coloneqq \{(x,y) \in (0,1)^2: y <  x \}$
the reference triangle.
Sets of the form $\{x = y\}$, $\{x = 0 \}$, $\{x = y\}$,
etc. refer to edges and diagonals of $\hS$ or $\hT$ and analogously 
$\{y \leq x\} = \{(x,y) \in \widehat{S} :  y \leq x\}$. 

For $q\in \N$, ${\mathbb P}_q = \operatorname{span} \{x^i y^j\,|\, i,j \ge 0, i+j \leq q\}$ 
denotes the space of polynomials of total degree $q$ and 
${\mathbb Q}_q = \operatorname{span} \{x^i y^j\,|\, 0 \leq i,j \leq q\}$ 
denotes the tensor product space of polynomial of maximum degree $q$ in each variable separately. 

For $x\in \Omega$, we recall $r(x) = \dist(x,\partial\Omega)$.
Finally, for $t>0$, we denote a $t$-neighborhood of $\partial \Omega$ 
by
\begin{equation*}
  S_t = \{x\in \Omega: r(x)< t\}.
\end{equation*}
%
\section{Analytic Regularity in Polygons with Straight Sides}
\label{sec:Reg}
We start by recapitulating the weighted spaces from \cite{FMMS21_983} 
used to describe the analytic regularity.

Recall that
$\Omega \subset \R^2$ is a bounded polygon with 
a finite number of straight sides, 
whose boundary $\partial\Omega$ is Lipschitz. 
By $\mathcal{V}$, we denote the set of vertices of the polygon 
$\Omega \subset \R^2$ and 
by $\mathcal{E}$ the set of its (open) edges. 
For $\mathbf{v} \in \mathcal{V}$ and $\mathbf{e} \in \mathcal{E}$, 
we define the distance functions 
\begin{align*} 
r_{\mathbf{v}}(x)\coloneqq|x - \mathbf{v}|, 
\qquad 
r_{\mathbf{e}}(x)\coloneqq\inf_{y \in \mathbf{e}} |x - y|, 
\qquad 
\rho_{\mathbf{v} \mathbf{e}}(x)\coloneqq r_{\mathbf{e}}(x)/r_{\mathbf{v}}(x). 
\end{align*} 
For each vertex $\mathbf{v} \in \mathcal{V}$, 
we denote by 
$\mathcal{E}_{\mathbf{v}}\coloneqq \{\mathbf{e} \in \mathcal{E}\,:\, \mathbf{v} \in \overline{\mathbf{e}}\}$
the set of all edges that meet at $\mathbf{v}$. 
For any $\mathbf{e} \in \mathcal{E}$, 
we define $\mathcal{V}_{\mathbf{e}}\coloneqq \{\mathbf{v} \in \mathcal{V}\,:\, 
\mathbf{v} \in \overline{\mathbf{e}}\}$ as set of endpoints of $\mathbf{e}$. 
For fixed, sufficiently small $\omegaeps> 0$ and for 
$\mathbf{v} \in \mathcal{V}$, $\mathbf{e} \in \mathcal{E}$, 
we define 
vertex, vertex-edge and edge neighborhoods by 
\begin{align}
\omegac^{\omegaeps} &\coloneqq \{x \in \Omega\,:\, r_{\mathbf{v}}(x) < \omegaeps 
                                     \quad \wedge \quad \rho_{\mathbf{v}\mathbf{e}}(x) \geq \omegaeps
\quad \forall \mathbf{e} \in \mathcal{E}_{\mathbf{v}}\}, 
\label{eq:omegv}
\\
\omegace^{\omegaeps} &\coloneqq \{x \in \Omega\,:\, r_{\mathbf{v}}(x) < \omegaeps 
                                     \quad \wedge \quad \rho_{\mathbf{v}\mathbf{e}}(x) < \omegaeps \},
\label{eq:omegve}
\\
\omegae^{\omegaeps} &\coloneqq \{x \in \Omega\,:\, r_{\mathbf{v}}(x) \geq \omegaeps  
                                     \quad \wedge \quad  r_{\mathbf{e}}(x) < \omegaeps^2
\quad \forall \mathbf{v} \in \mathcal{V}_{\mathbf{e}}\}.
\label{eq:omege}
\end{align}
Fig.~\ref{fig:vertex-notation}, taken from \cite{FMMS21_983}, 
illustrates this notation near a vertex $\mathbf{v} \in \mathcal{V}$ of the polygon. 
Throughout the paper, we will assume that $\omegaeps$ is small enough so
that $\omegac^{\omegaeps} \cap \omegacprime^{\omegaeps} = \emptyset$ for all $\mathbf{v} \neq
\mathbf{v}'$, that $\omegae^{\omegaeps} \cap \omegaeprime^{\omegaeps} = \emptyset$  for all
$\mathbf{e}\neq \mathbf{e}'$ and $\omegace^{\omegaeps}\cap \omegaceprime^{\omegaeps} = \emptyset$ for
all $\mathbf{v}\neq \mathbf{v}'$ and all $\mathbf{e}\neq \mathbf{e}'$. We will
also drop the superscripts $\omegaeps$  unless strictly necessary.

\begin{figure}[ht]
\begin{center}
\begin{tikzpicture}[scale=1.7]
  \def\R{3}
  \def\A{60}
\draw ({3/4*\R*cos(\A)-.1},{3/4*\R*sin(\A)}) node[above]{$\mathbf{e}'$};
\draw (0,0) node {\textbullet} node[left] {$\mathbf{v}$}; 
  \draw[-] (0, 0) -- ({\R*cos(\A)}, {\R*sin(\A)});
  \draw[-] (0, 0) -- (\R, 0); 
  \draw[dashed] (0, 0) -- ({2/3*\R*cos(\A/3)},  {2/3*\R*sin(\A/3)}); 
  \draw[dashed] (0, 0) -- ({2/3*\R*cos(2*\A/3)},{2/3*\R*sin(2*\A/3)}); 
  \draw[dashed] ({2/3*\R*cos(\A/3)}, {2/3*\R*sin(\A/3)}) -- (\R,{2/3*\R*sin(\A/3)}); 
  \draw[dashed] %
  ({2/3*\R*cos(2*\A/3)}, {2/3*\R*sin(2*\A/3)}) %
  -- ({2/3*\R*cos(2*\A/3) + \R*(1-2/3*cos(\A/3))*cos(\A)}, {2/3*\R*sin(2*\A/3) + \R*(1-2/3*cos(\A/3))*sin(\A)});
\draw (3/4*\R, 0) node [below]{$\mathbf{e}$} ;
\draw (3/8*\R ,0.05) node [above]{$\omega_{\mathbf{v}\mathbf{e}}$} ;
\draw (7/8*\R, 0.15) node [above]{$\omega_{\mathbf{e}}$} ;
\draw ({\R*cos(\A/2)}, {\R*sin(\A/2)}) node {$\Omega_{\mathrm{int}}$} ;
\draw ({3/8*\R*cos(\A/2)},{3/8*\R*sin(\A/2)-0.12}) node [above]{$\omega_{\mathbf{v}}$}; 
\draw ({3/8*\R*cos(3*\A/4)},{3/8*\R*sin(3*\A/4)}) node [above]{$\omega_{\mathbf{v}\mathbf{e}'}$} ;
\draw ({7/8*\R*cos(3*\A/4)-0.25},{7/8*\R*sin(3*\A/4)}) node [above]{$\omega_{\mathbf{e}'}$} ;
\draw [dashed,domain=0:\A] plot ({2/3*\R*cos(\x)}, {2/3*\R*sin(\x)});
\end{tikzpicture}
\end{center}
\caption{\label{fig:vertex-notation} Notation near vertex $\mathbf{v} \in \mathcal{V}$.}
\end{figure}
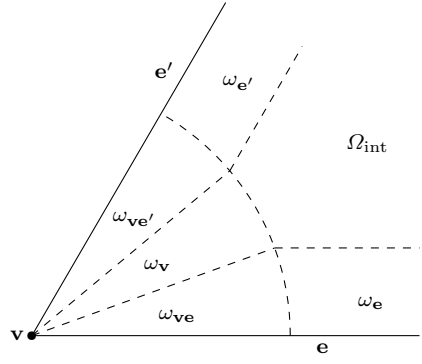

%
The polygon $\Omega$ may be decomposed 
into sectoral neighborhoods of vertices $\mathbf{v}$, 
which are unions of vertex-neighborhoods $\omega_{\mathbf{v}}$ 
and vertex-edge neighborhoods $\omega_{\mathbf{ve}}$
(as depicted in Fig.~\ref{fig:vertex-notation}), 
edge neighborhoods $\omega_{\mathbf{e}}$ (that are properly separated from vertices $\mathbf{v}$),
and an interior part  $\Omega_{\rm int}$, 
i.e., we may write
\begin{align*}
 \Omega = \bigcup_{\mathbf{v} \in \mathcal{V}}\left( \omega_{\mathbf{v}} 
        \cup \bigcup_{\mathbf{e} \in \mathcal{E}_{\mathbf{v}}} \omega_{\mathbf{ve}} \right) 
        \cup \bigcup_{\mathbf{e} \in \mathcal{E}} \omega_{\mathbf{e}} 
        \cup \Omega_{\rm int}.
\end{align*}
Each sectoral and edge neighborhood may have a different value $\omegaeps$,
but we shall work with one common (positive) value for all neighborhoods.
The set $\Omega_{\rm int}\subset\Omega $ has a positive distance from the boundary
$\partial \Omega$.
\bigskip

In a neighborhood $\omega_{\mathbf{e}}$ or $\omega_{\mathbf{v}\mathbf{e}}$, 
we denote by 
$\mathbf{e}_{\parallel}$ and $\mathbf{e}_{\perp}$ unit vectors 
such that $\mathbf{e}_{\parallel}$ is tangential to $\mathbf{e}$ 
and $\mathbf{e}_{\perp}$ is normal to $\mathbf{e}$. 
We introduce the differential operators 
\begin{align*}
D_{x_\parallel} v &\coloneqq \mathbf{e}_{\parallel} \cdot \nabla_x v, 
& D_{x_\perp} v &\coloneqq \mathbf{e}_{\perp} \cdot \nabla_x v
\end{align*}
corresponding to differentiation in the tangential and normal direction. 
Higher order tangential and normal derivatives 
in $\omega_{\mathbf{e}}$ or $\omega_{\mathbf{v}\mathbf{e}}$ 
are defined by 
$D_{x_\parallel}^j v \coloneqq D_{x_\parallel}(D_{x_\parallel}^{j-1} v)$ and
$D_{x_\perp}^j v \coloneqq D_{x_\perp}(D_{x_\perp}^{j-1} v)$ for $j>1$. \bigskip

The analytic regularity result in weighted local norms 
is \cite[Thm.~{2.1}]{FMMS21_983}.
\begin{theorem}\label{thm:mainresult}
Let $\Omega \subset \R^2$ be a bounded polygonal Lipschitz domain.
  Let the data $f \in C^{\infty}(\overline{\Omega})$ satisfy
with a constant $\gamma_f>0$
  \begin{equation}
    \label{eq:analyticdata}
\forall j \in \N_0\colon \quad 
    \sum_{\alpham = j} \|\dalpha f\|_{L^2(\Omega)} \leq \gamma_f^{j+1}j^j. 
  \end{equation}
Let $u$ be the solution of \eqref{eq:weakform}.
Let $\mathbf{v} \in \mathcal{V}$, $\mathbf{e} \in \mathcal{E}$ and
$\omega_{\mathbf{v}}$, $\omega_{\mathbf{ve}}$, $\omega_{\mathbf{e}}$ be 
fixed vertex, vertex-edge and edge-neighborhoods.
Then, there is $\gamma > 0$ depending only on $\gamma_f$, $s$, and $\Omega$ such that for 
every $\varepsilon>0$ there exists $C_\varepsilon>0$ (depending only on $\varepsilon$ and $\Omega$) 
such that the following holds: 
\begin{enumerate}
\item[(i)] 
For all $\alpha \in \mathbb{N}^2_0$
  \begin{equation}
 \label{eq:analytic-u-c-all}
 \norm{r_{\mathbf{v}}^{p-1/2-s+\varepsilon} \dalpha u}_{L^2(\omega_{\mathbf{v}})} 
 \leq C_{\varepsilon} \gamma^{{\alpham}+1}{\alpham}^{\alpham}.
  \end{equation}
\item[(ii)]
   For all $(\pperp, \ppar)\in \mathbb{N}^2_0$ it holds, 
   with $p=\pperp+\ppar$, that
\begin{align}  
 \norm{r_{\mathbf{e}}^{\pperp-1/2-s+\varepsilon} D^{\pperp}_{x_\perp} D^{\ppar}_{x_{\parallel}} u }_{L^2(\omegae)} 
 & \leq C_{\varepsilon} \gamma^{p+1} p^p,
\label{eq:analytic-u-e-all}\\
  \norm{r_{\mathbf{e}} ^{\pperp-1/2-s + \varepsilon} r_{\mathbf{v}} ^{\ppar+\varepsilon} D^{\pperp}_{x_\perp} D^{\ppar}_{x_\parallel} u}_{L^2(\omegace)}
 & \leq C_{\varepsilon} \gamma^{p+1} p^p.
\label{eq:analytic-u-ce-all}
\end{align}
\item[(iii)] 
In the interior $\Omega_{\rm int}$,  for 
all $\alpha \in \N_0^2$,
\begin{equation}
  \label{eq:analytic-u-int}
 \norm{{\dalpha} u }_{L^2(\Omega_{\rm int})} \leq \gamma^{\alpham+1}\alpham^{\alpham}.
\end{equation}
\end{enumerate}
\end{theorem}
\section{Embedding into weighted integer order space}
\label{sec:LoclNrm}

The nonlocal nature of the $\tH^s(\Omega)$-norm \eqref{eq:FracNrm}
is well-known to obstruct the common FE-approximation strategy to obtain
global error bounds by adding scaled, local error estimates on subdomains.
Accordingly, as proposed in \cite[Sec.~{3.4}]{FMMS-hp1d},
we localize this norm via an embedding into 
a weighted integer order space. 
While such embeddings are known (e.g. \cite[Sec. 3.4]{triebel95}), 
we provide a short proof to render the exposition self-contained.

Recall $r(x):=\operatorname{dist}(x,\partial\Omega)$ for $x\in \Omega$.
For $\beta \in [0,1)$, 
and an open set $\omega\subseteq \Omega$
denote by 
$H^1_\beta(\omega)$ the local Sobolev space 
defined via the weighted norm 
$\| \cdot \|_{H^1_\beta(\omega)}$ given by
\begin{equation}\label{eq:H1b}
\| v \|^2_{H^1_\beta(\omega)}
:= 
\| r^\beta \nabla v \|_{L^2(\omega)}^2
+ 
\| r^{\beta-1} v    \|_{L^2(\omega)}^2
.
\end{equation}
\begin{proposition}[\protect{\cite[Lem.~8]{FMMS-hp1d}}]
\label{prop:LocNrm}
Let $\Omega \subset \R^d$ denote a bounded domain with
Lipschitz boundary $\partial\Omega$,
and assume $\sigma \in (0, 1]$.
Denote by $H^1_\beta(\Omega)$ the closure of $C_0^\infty(\Omega)$
with respect to the norm $\| \cdot \|_{H^1_\beta(\Omega)}$ in \eqref{eq:H1b}.
\newline
Then, for all $\beta \in [0,1-\sigma)$, 
$H^1_\beta(\Omega)$ is continuously embedded into $\tH^\sigma(\Omega)$, 
and 
there exists a constant $C_{\beta,\sigma}(\Omega)>0$ 
such that
\begin{equation}\label{eq:LocNrm}
\forall v\in H^1_\beta(\Omega): \quad 
\| v \|_{\tH^\sigma(\Omega)}
\leq 
C_{\beta,\sigma} \| v \|_{H^1_\beta(\Omega)} 
\;.
\end{equation}
The estimate \eqref{eq:LocNrm} remains valid in the limit case $(\sigma,\beta) = (1,0)$. 
%
\end{proposition}
\begin{proof}
We present the argument from the univariate case \cite[Lem.8]{FMMS-hp1d}, 
with the minor adaptations to the present setting.
For Banach spaces $X_1 \subset X_0$ with continuous injection,
and for $v\in X_0$, $t>0$, the $K$-functional is given by
$K(v,t; X_0, X_1):= \inf_{w \in X_1} \|v - w\|_{X_0} + t\|w\|_{X_1}$.
For $\theta \in (0,1)$ and $q\in [1,\infty)$,
the interpolation spaces (e.g. \cite[Chap.1.3]{triebel95})
$X_{\theta,q}:= (X_0,X_1)_{\theta,q}$ are given by the norm
\begin{equation}
  \label{eq:interp-norm}
\|v\|^q_{X_{\theta,q}}
:=
\int_{t=0}^\infty \left(t^{-\theta} K( v,t; X_0, X_1)\right)^q\frac{dt}{t}.
\end{equation}
We now choose $X_0 = L^2(\Omega)$ and $X_1 = H^1_0(\Omega)$ and fix $\beta \in (0,1-\sigma)$.
We note that the function $r$ is Lipschitz.
For each $t > 0$ sufficiently small, we may choose $\chi_t \in C^\infty(\R)$ 
such that
$\chi_t\circ r \equiv 0$ on the strip $S_{t/2}$
and
$\chi_t\circ r \equiv 1$ on $\Omega \setminus S_t$ 
as well as
$\|\nabla^{j} (\chi_t\circ r)\|_{L^\infty(\R^2)} \leq C t^{-j}$, $j \in \{0,1\}$.
Decomposing 
$v = (\chi_t\circ r) v +  (1-(\chi_t\circ r)) v$, 
we have $(\chi_t\circ r )v \in H^1_0(\Omega)$ 
and $(1 - (\chi_t\circ r)) v \in L^2(\Omega)$ 
for $v \in H^1_{\beta}(\Omega)$.

A calculation 
shows that there exists a constant $C>0$ such that
\begin{align*}
\forall v\in C^\infty_0(\Omega): \quad 
\|\nabla ((\chi_t\circ r) v)\|_{L^2(\Omega)} 
& \leq C t^{-\beta} \|v\|_{H^1_{\beta}(\Omega)}, 
\\
\|(1 - (\chi_t\circ r)) v\|_{L^2(\Omega)} 
& 
\leq C t^{1-\beta} \|r^{\beta - 1} v\|_{L^2(\Omega)}.
\end{align*}
This implies that
$K(v,t; X_0, X_1) \leq C t^{1- \beta} \|v\|_{H^1_{\beta}(\Omega)}$ for small $t>0$.
Since $X_1 \subset X_0$,
replacing the integration
limit $\infty$ in \eqref{eq:interp-norm}
by a finite number $T$
leads to an equivalent norm \cite[Chap.~6, Sec.~7]{devore93}. Hence,
  \begin{equation*}
    \| v \|_{X_{\sigma, 2}}^2 \simeq \int_{0}^T \left[ t^{-\sigma}K(v, t; X_0, X_1) \right]^2\frac{dt}{t} 
    \leq C \|v\|_{H^1_{\beta}(\Omega)}^2 \int_0^T t^{1-2\beta-2\sigma} dt, 
  \end{equation*}
  and the latter integral is bounded for all $\beta < 1-\sigma$.
  We conclude by remarking that $\tH^\sigma(\Omega) = X_{\sigma, 2}$ with
 equivalent norms \cite[Prop.~{4.1} and Thm.~{4.10}]{Brasco2019}.

The validity of the assertion in the limiting case $(\sigma,\beta) = (1,0)$ 
    follows from \cite[Thm.~{1.4.4.3}]{Grisvard}.
  \qed
\end{proof}
%
\section{Geometrically refined meshes}
\label{sec:mesh}
%
We review here briefly the 
patch-wise construction of geometrically refined meshes
from \cite{banjai-melenk-schwab19-RD,BMS20_2880}.
We admit both triangular
and 
quadrilateral
elements $K\in \calT$,
but \emph{do not} assume shape regularity:
\emph{anisotropic, geometric mesh refinement towards $\partial \Omega$} 
is essential to resolve edge singularities
that are generically present in solutions of fractional PDEs.
%

\subsection{Macro triangulation. Mesh patches}
\label{sec:geo-bdy-layer-mesh}
\begin{figure}
\psfragscanon
\psfrag{T}{$T$}
\psfrag{B}{$B$}
\psfrag{M}{$M$}
\psfrag{C}{$C$}
\begin{overpic}[width=0.25\textwidth]{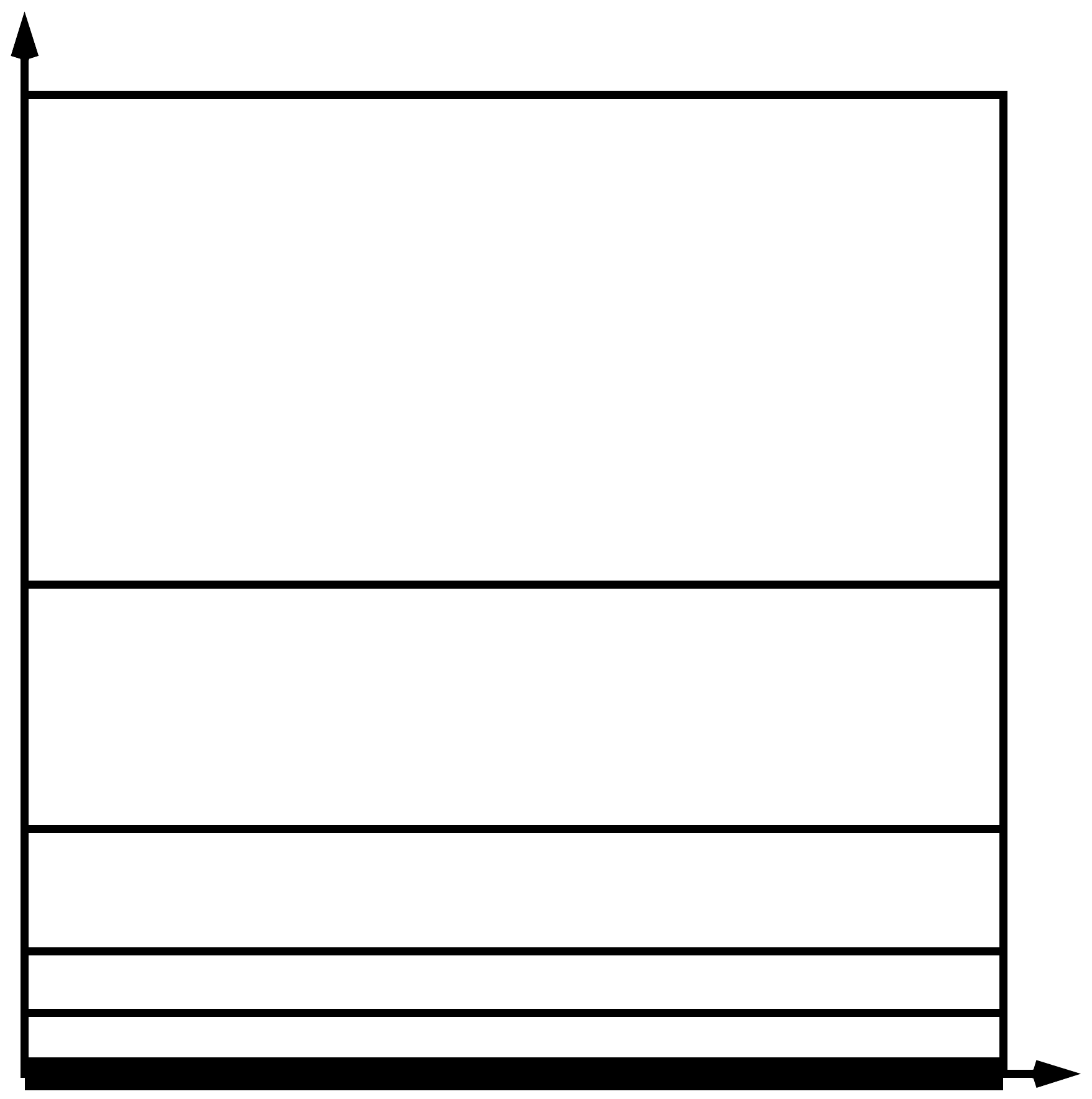}
\put(45,70){$ \Tedg$}
\put(100,-5){\small $\tx$}
\put(-7,95){\small $\ty$}
\end{overpic}
\hfill 
\begin{overpic}[width=0.25\textwidth]{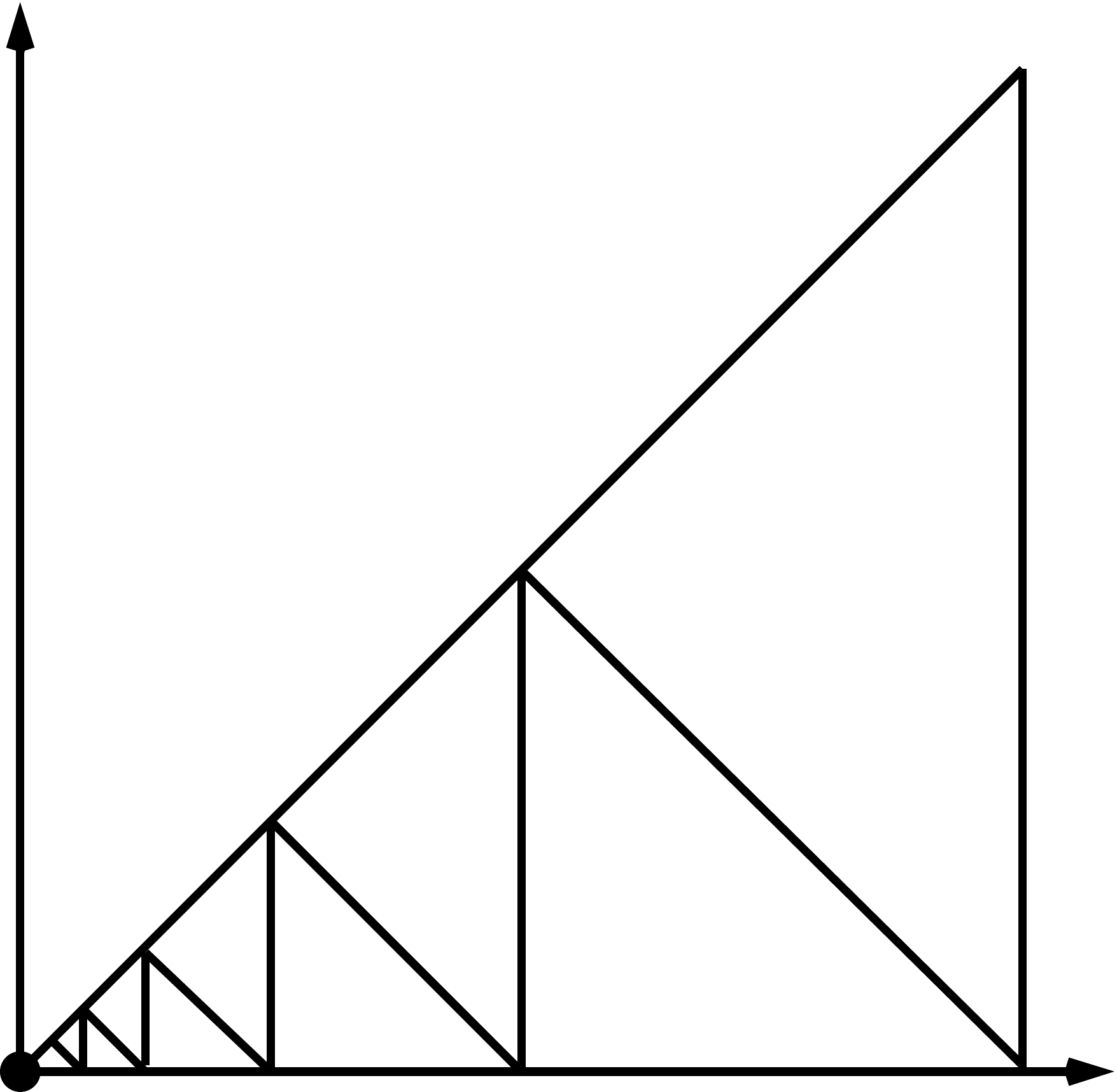}
\put(60,50){$ \check{\calT}^{\Co,L}_{geo,\sigma}$}
\put(100,-5){\small $\tx$}
\put(-7,95){\small $\ty$}
\end{overpic}
\hfill 
\begin{overpic}[ width=0.25\textwidth]{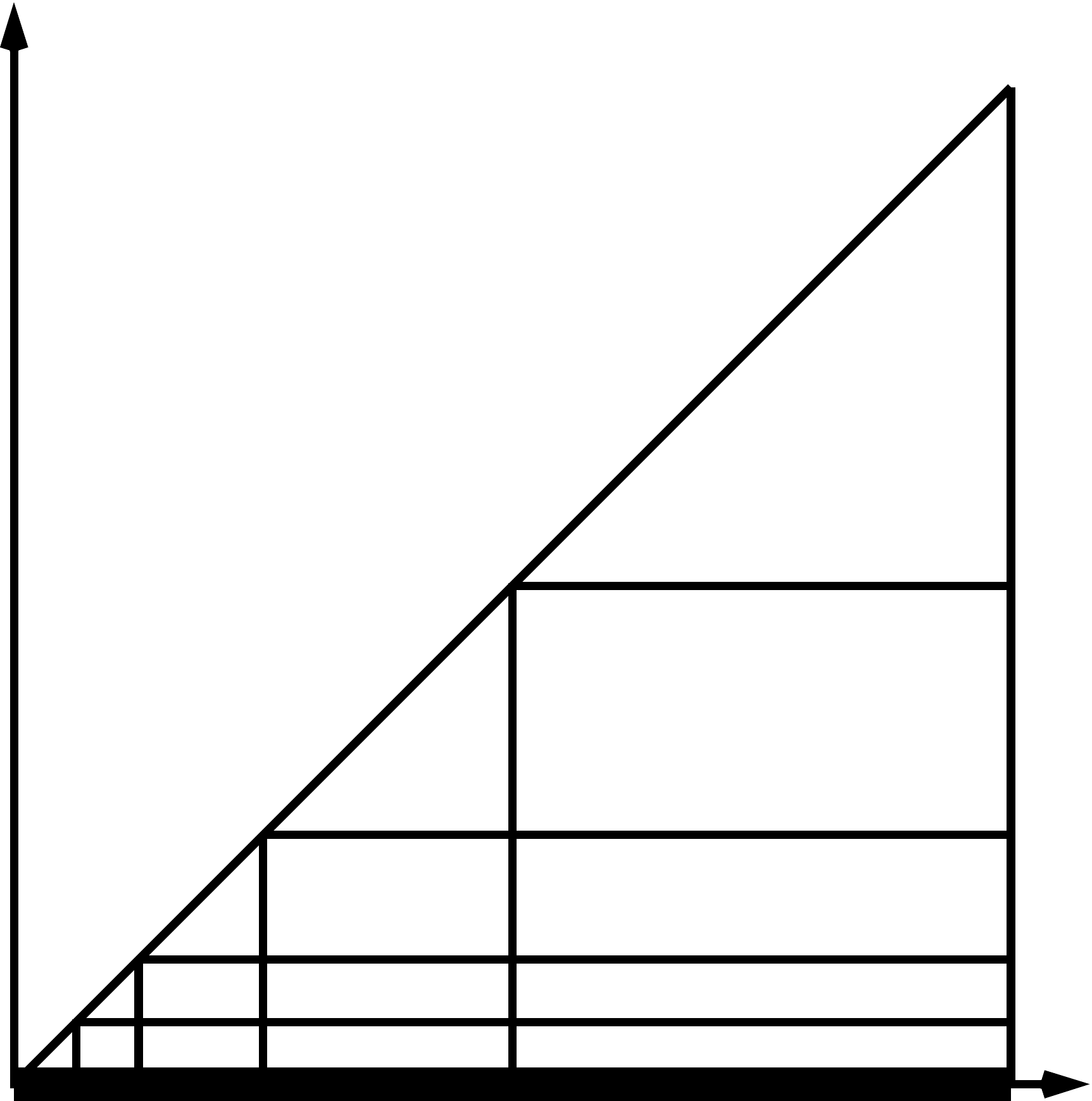}
\put(60,30){$\Tce$}
\put(-7,95){\small $\ty$}
\put(100,-5){\small $\tx$}
\end{overpic}
\\[1.5em]
\centering
\begin{overpic}[width=0.25\textwidth]{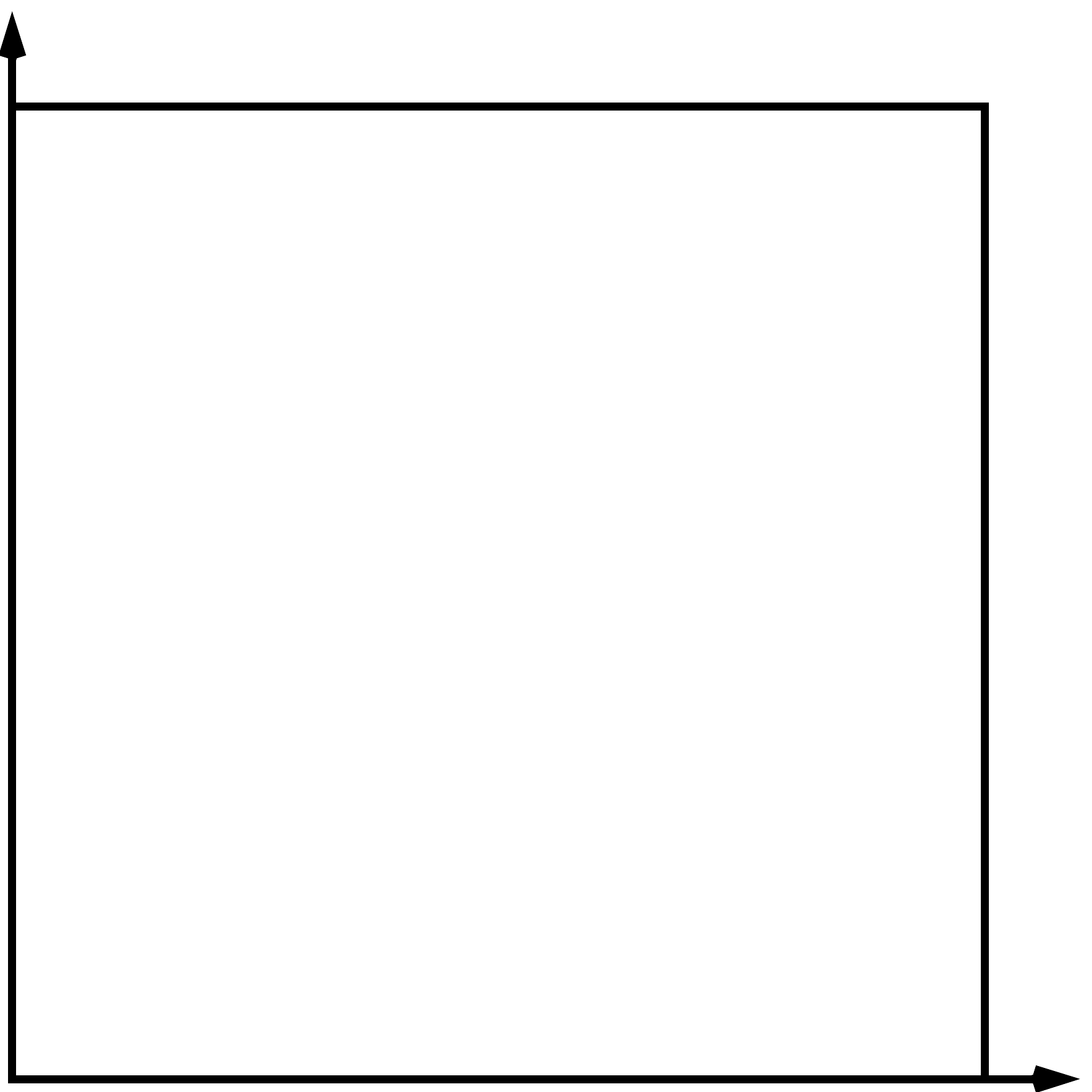}
\put(25,70){trivial patch}
\put(100,-5){\small $\tx$}
\put(-7,95){\small $\ty$}
\end{overpic}
\psfragscanoff
\caption{
\label{fig:patches} 
Catalog ${\mathfrak P}$ of reference refinement patterns.
Top row: reference edge patch $\Tedg$
with $L$ layers of geometric refinement towards $\{\ty=0\}$;
reference vertex patch $\Tcor$
with $L$ layers of geometric refinement towards $(0,0)$;
vertex-edge patch $\Tce$
with $L$ layers of refinement towards $(0,0)$ 
and $L$ layers of refinement towards $\{\ty=0\}$. 
Bottom row: trivial patch.
Geometric entities shown in boldface indicate parts of $\partial \widehat S$
that are mapped to $\partial\Omega$.
These patch meshes are transported into the polygon $\Omega$
via patch maps $F_{\KM}$.
}%
\end{figure}
\begin{figure}
  \centering
  \begin{subfigure}{.5\linewidth}
    \centering
    \includegraphics[width=.6\textwidth]{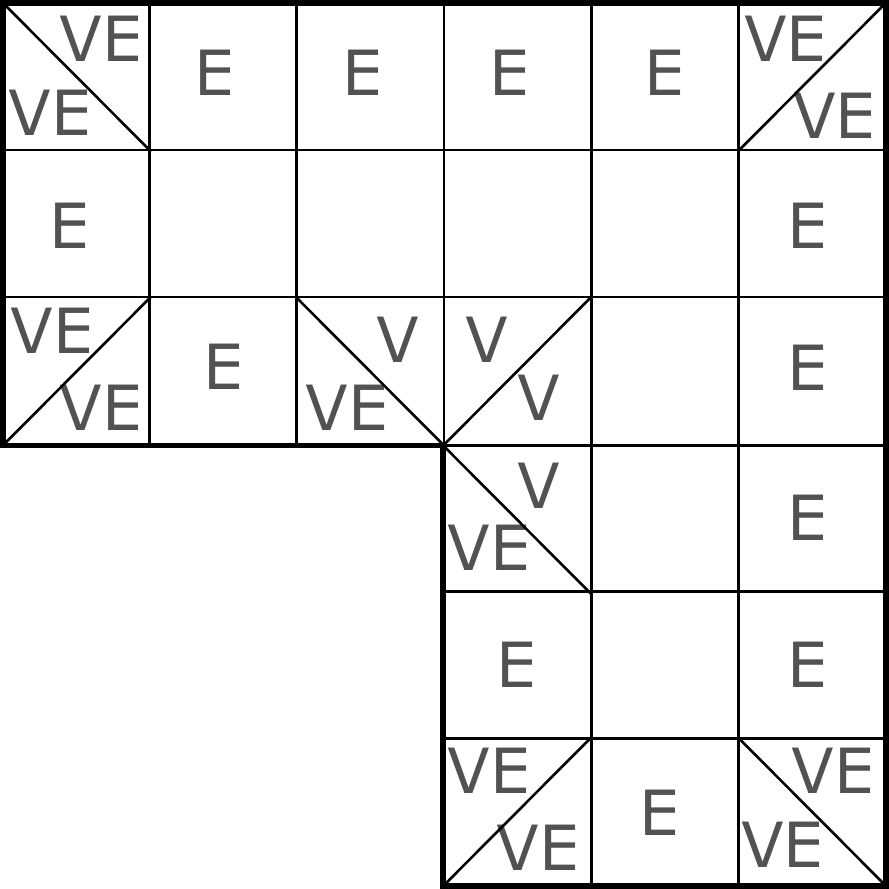}
  \end{subfigure}%
  \begin{subfigure}{.5\linewidth}
    \centering
    \includegraphics[width=.9\textwidth]{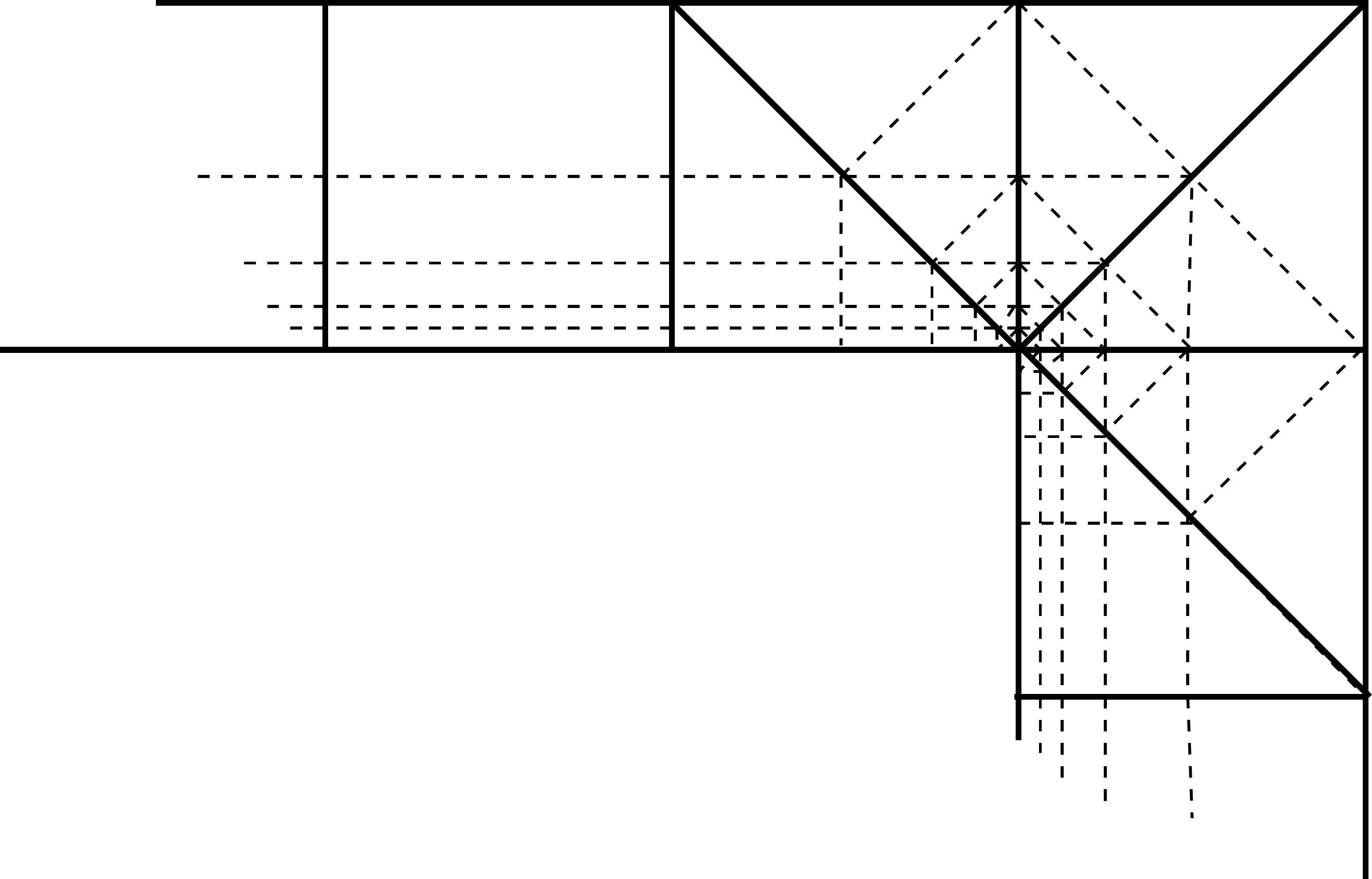}
  \end{subfigure}


\caption{
\label{fig:patch-examples} 
Patch arrangement in $\Omega$.
Left panel: example of L-shaped domain decomposed into patches 
($\Co$, $\mathsf{E}$, $\CE$ indicate Vertex, Edge, Vertex-Edge patches,
empty squares signify trivial patches). 
Right panel: Zoom-in near the reentrant corner $\mathbf{v}$. 
Solid lines indicate patch boundaries, dashed lines mesh lines.
}
\end{figure}
We recapitulate the $hp$-FE approximation theory 
on geometrically refined meshes generated 
as push-forwards of a small number of so-called \emph{mesh patches}, 
similar to those introduced (for the $hp$-approximation of singularly perturbed,
linear elliptic boundary value problems) in \cite[Sec.~{3.3.3}]{melenk02} and \cite{FstmnMM_hpBalNrm2017}.
These mesh families are based on 
a \emph{fixed macro-triangulation}
${\mathcal T}^{\M}$ 
of the domain $\Omega$.
The macro-triangulation ${\mathcal T}^{\M}$ 
consists of mapped triangles and quadrilaterals $\KM$ 
which are endowed with \emph{patch maps} 
(to be distinguished from the actual element maps) 
$F_{\KM}:\hS \rightarrow \KM$, for quadrilateral patches,
and 
$F_{\KM}:\hT \rightarrow \KM$, for triangular patches,
that satisfy the usual compatibility conditions\footnote{
${\mathcal T}^{\M}$ does not have hanging nodes and, 
for any two distinct elements $\KM_1, \KM_2 \in {\mathcal T}^{\M}$ that
share an edge $e$, their respective element maps induce compatible parametrizations
of $e$ (cf., e.g.,  \cite[Def.~{2.4.1}]{melenk02} for the precise conditions).
}.
Each element of the fixed macro-triangulation ${\mathcal T}^{\M}$
is further subdivided according to one of the 
refinement patterns in Definition~\ref{def:admissible-patterns} below
(see also \cite[Sec.~{3.3.3}]{melenk02} or \cite{FstmnMM_hpBalNrm2017}). 
The actual triangulation is then obtained by transplanting refinement patterns 
on the reference patch into the physical domain $\Omega$
by means of the patch maps $F_{\KM}$ of the macro-triangulation. 
That is, 
for any element $K\in {\mathcal T}$, at refinement level $L\in \N$, the element map $F^L_K$ is the 
concatenation of an affine map---which realizes the mapping from the 
reference square or triangle to the elements 
in the patch refinement pattern and will be denoted by $A_K^L$---
and the patch map (denoted by $F_{\KM}$), 
i.e., $F^L_K = F_{\KM} \circ A^L_K : \hat K \rightarrow K$. 
We introduce the refinement patterns, see also \cite{Holm2008} and \cite[Def.~{2.1}]{banjai-melenk-schwab19-RD}.

\begin{definition}[Catalog ${\mathfrak P}$ of refinement patterns]
\label{def:admissible-patterns}
Given $\sigma \in (0,1)$, $L \in {\mathbb N}_0$ 
the catalog ${\mathfrak P}$ consists of the following patterns:
\begin{enumerate}
\item
The \emph{trivial patch}: 
The reference square $\widehat S = (0,1)^2$ is not further refined. 
The corresponding triangulation of $\widehat S$ consists of the single element: 
$\check{\mathcal T}^{\mathrm{trivial}} = \{\widehat S\}$. 
\item
The \emph{geometric edge patch $\Tedg$}: 
$\widehat S$ is refined anisotropically towards 
$\{\yh =0\}$ into $L$ elements as depicted
in Fig.~\ref{fig:patches} (top left). 
The mesh 
$\Tedg$ 
is characterized by the nodes $(0,0)$, $(0,\sigma^i)$, $(1,0)$, $(1,\sigma^i)$,
$i=0,\ldots,L$ and the corresponding rectangular elements generated by these nodes.
\item
The \emph{geometric vertex patch $\Tcor$}: 
$\hT$ is refined isotropically towards $(0,0)$ as depicted 
in Fig.~\ref{fig:patches} (top middle). 
The reference geometric vertex patch mesh $\Tcor$
in $\hT$ with geometric refinement towards $(0,0)$ and $L$ layers
is given by triangles determined by the nodes
$(0,0)$, $(\sigma^i,0)$, and
$(\sigma^i,\sigma^i)$, $i=0,\ldots,L$. 
\item
  The \emph{vertex-edge patch $\Tce$}:
  The triangulation, depicted in Fig.~\ref{fig:patches} (top right),
  consists of both anisotropic elements and isotropic elements.
It is given by the nodes 
$(0,0)$, $(\sigma^i,0)$, $(\sigma^i,\sigma^{j})$, $0 \leq i \leq L$, $i \leq j \leq L$
and consists of anisotropic rectangles and uniformly shape-regular triangles. 
\end{enumerate}
\end{definition}
%
%
\subsection{Geometric boundary-refined mesh $\Tg$}
\label{sec:GeomBdryMes}
We now define the global, boundary-refined meshes $\Tg$,
which will be used in the definition of the FE space \eqref{eq:S^q_0}.
These meshes are built by assembling possibly anisotropic, geometric 
patch partitions from the catalog ${\mathfrak P}$ in Definition~\ref{def:admissible-patterns}.
To ensure inter-patch compatibility, all partitions from ${\mathfrak P}$ are
taken with the same values of $\sigma$ and $L$. The resulting partitions of $\Omega$
are regular, and feature anisotropic, geometric refinement
towards the edges $\mathbf{e} \subset \partial\Omega$
and isotropic geometric refinement towards the vertices $\mathbf{v}\subset \partial\Omega$.
\begin{definition} 
[geometric boundary-refined mesh, \protect{\cite[Def.~{2.3}]{banjai-melenk-schwab19-RD}}]
\label{def:bdylayer-mesh} 
Let ${\mathcal T}^{\M}$ be a fixed macro-triangulation 
consisting of quadrilateral or triangular patches with 
bilinear and affine patch maps.
Patch-refinement patterns are specified in terms of 
parameters $\sigma$ and $L$.


Given $\sigma \in (0,1)$, $L \in {\mathbb N}_0$, 
$\Tg$ is called a \emph{geometric boundary-refined mesh},
if the following conditions hold:
\begin{enumerate}
\item 
$\Tg$ is obtained by refining each element 
$\KM \in {\mathcal T}^{\M}$
according to the finite catalog ${\mathfrak P}$ 
of patch-refinement patterns as
specified in Definition~\ref{def:admissible-patterns}.
\item 
$\Tg$ is a regular partition of $\Omega$.
 i.e., it does not have hanging nodes. 
Since the element maps for the refinement patterns
are assumed to be affine or bilinear, this requirement ensures that the 
resulting triangulation satisfies 
\cite[Def.~{2.4.1}]{melenk02}.  
\end{enumerate}
For each macro-patch $\KM \in {\cT}^{\M}$, 
exactly one of the following cases is possible: 
\begin{enumerate}
\setcounter{enumi}{2}
\item 
\label{item:def-geo-3}
$\overline{\KM} \cap \partial\Omega = \emptyset$. 
Then, the trivial patch is selected as the reference patch. 
We denote $\cT^{\M}_{\mathrm{int}} $ the set of such macro-elements.
\item 
\label{item:def-geo-4} 
$\overline{\KM} \cap \partial\Omega = \{\bP\}$ is a single point,
where $\bP$ can be a vertex of $\Omega$ or a point on the boundary. 
The refinement pattern is the
vertex patch $\Tcor$ with 
$L$ layers of geometric mesh refinement towards the origin $\bO$;
it is assumed that $F_{\KM}(\bO)  = \bP \in \partial\Omega$. 
We denote $\cT^{\M}_{\Co} $ the set of such macro-elements.
\item 
\label{item:def-geo-5}
$\overline{\KM} \cap \partial \Omega = \overline{e}$ for an edge $e$ of 
$\KM$ and neither  endpoint of $e$ is a vertex of $\Omega$. Then,
the refinement pattern is the edge patch 
$\Tedg$ and additionally 
$F_{\KM}(\{\ty = 0\}) \subset \partial\Omega$. 
We denote $\cT^{\M}_{\Edg} $ the set of such macro-elements.
\item 
\label{item:def-geo-6}
$\overline{\KM} \cap \partial \Omega = \overline{\mathbf{e}}$ for an edge $\mathbf{e}$ of 
$\KM$ and exactly one endpoint of $\mathbf{e}$ is a vertex $\mathbf{v}$ of $\Omega$. 
The refinement pattern is the vertex-edge patch
$\Tce$ and additionally 
$F_{\KM}(\{\ty = 0\}) \subset \partial\Omega$ 
as well as
$F_{\KM}(\bO) = \mathbf{v}$. 
We denote $\cT^{\M}_{\CE} $ the set of such macro-elements.
\end{enumerate}
\end{definition}
We assume that $F_{\KM}$ is bilinear for all $\KM \in \cT^\M_{\mathrm{int}}$,
 and that it is affine for all $\KM\in \cT^\M_{\Edg}\cup \cT^\M_{\Co}\cup \cT^\M_{\CE}$.

%
\begin{example}
Fig.~\ref{fig:patch-examples} 
shows a so-called ``$L$-shaped domain'' with macro triangulation and 
patch-refinement patterns in the vicinity of a re-entrant corner $\mathbf{v}$.
\eremk
\end{example}
%
%
\section{$hp$-Approximation on geometric boundary-refined meshes}
\label{sec:hp-approx}
The exponential convergence of $hp$ approximations 
for functions $u\in \tH^s(\Omega)$ that satisfy 
the weighted analytic regularity
\eqref{eq:analytic-u-c-all}--\eqref{eq:analytic-u-int}
will be developed in several steps. 
As is customary in proofs of FE error bounds,
we shall obtain exponential convergence from 
the quasioptimality \eqref{eq:QuasiOpt} by constructing
$v_N = \Pi_N u$ in a subspace $V_N\subset \tH^s(\Omega)$
which is designed to exploit 
\eqref{eq:analytic-u-c-all}--\eqref{eq:analytic-u-int}.
Specifically, we shall use an $hp$-patch framework
similar to the one developed in \cite{banjai-melenk-schwab19-RD,BMS20_2880}
for exponentially convergent approximations of solutions to
singular perturbation problems and of spectral fractional 
diffusion in $\Omega$. 
We recapitulate in Section~\ref{S:hpFEMOmega}
this $hp$-approximation framework.

\subsection{$hp$-FE Spaces in $\Omega$}
\label{S:hpFEMOmega}
On the geometric partitions $\Tg$ introduced in Section
\ref{sec:mesh},
we consider Lagrangian finite elements of uniform polynomial degree $q\geq 1$, 
i.e.,
we choose the global finite element space $\VLq$ in \eqref{eq:QuasiOpt} as 
\begin{equation}
\label{eq:S^q_0}
\VLq\coloneqq \mathcal{S}^q_0(\Omega,\Tg)
\coloneqq 
\left \{ v \in C(\overline{\Omega}): v|_{K}\circ F^L_K \in \mathbb{V}_{q}(\widehat K) 
\quad 
\forall K \in \Tg, \ v|_{\partial \Omega} = 0 \right \}.
\end{equation}
Here,
for $q \geq 1$, the local
polynomial space is
\begin{equation*}
  \mathbb{V}_q(\widehat K) =
  \begin{cases}
    {\mathbb P}_q &\text{if } \widehat K = \widehat T, \\
    {\mathbb Q}_q &\text{if } \widehat K = \widehat S.
  \end{cases}
\end{equation*}

%
\subsection{Definition of the $hp$-interpolation operator  $\Pi_q^L$}
\label{S:DefhpInt}
The global $hp$-interpolator 
$\PiLq: H^1_\beta(\Omega) \to \VLq$ will be obtained 
by assembling local Gauss-Lobatto-Legendre (GLL) 
interpolants in the reference patches.
The global error estimate will follow from 
addition of patchwise error bounds in $H^1_\beta$ 
in the reference patches.
Addition of element-wise and patch-wise error bounds 
is possible due to the locality of the $H^1_\beta$-norm.
In possibly anisotropic quadrilateral elements,
the GLL interpolants are generated by
tensorization of univariate GLL interpolants.
We review their definition and properties 
briefly in Appendix~\ref{sec:approx-reference-element}.
Recall that all triangular elements are shape-regular. 
Only quadrilateral elements may be anisotropic. 
\subsubsection{Definition of the $hp$-interpolator $\widetilde \Pi^L_q$ on reference patches}
\label{sec:approx-reference-patches}
The $hp$-approximation operators on reference patches 
are obtained by assembling elementwise 
GLL-interpolants (cf. \cite[Eqn.(3.6)]{banjai-melenk-schwab19-RD}).
Recalling $A^L_\Kt: \widehat K \rightarrow \Kt = A^L_\Kt(\widehat K) \in \Tgeneric\cup\Ttrivial$ with $\bullet\in \{ \Co,\Edg,\CE \}$
the \emph{affine bijection between the reference element $\widehat K$
and the corresponding element $\Kt$ on the reference patch}, we set 
\begin{equation}
\label{eq:tildePiq} 
(\widetilde \Pi^L_q)|_{\Kt} v:= 
\begin{cases} 
\widehat\Pi_q^\triangle (v \circ A^L_\Kt) & \text{ if $\Kt$ is a triangle,} \\
\widehat\Pi_q^{\Box} (v \circ A^L_\Kt) & \text{ if $\Kt$ is a rectangle}. 
\end{cases} 
\end{equation}
The elemental GLL interpolators $\widehat \Pi_q^\triangle$ and $\widehat \Pi_q^{\Box}$ 
defined in Lemmata~\ref{lemma:hat_Pi_infty}, \ref{lemma:hat_Pi_1_infty}
commute with trace operators on the edges of $\Kh$.
This ensures global $H^1_\beta$-conformity of the 
reference patch interpolator $\widetilde \Pi^L_q$.
\subsubsection{Definition of the global $hp$-interpolator $\Pi_q^L$}
\label{sec:DefGlobhp}
With the $hp$ patch-interpolants in \eqref{eq:tildePiq} in place, 
the global $hp$-interpolator $\Pi^L_q$
is assembled from elementwise projectors on an element $K$ 
via
\begin{align*}
(\Pi^L_q u)|_K \circ F^L_K & :=
\begin{cases}
\PiT (u \circ F^L_K) & \mbox{ if $K$ is a triangle}, \\
\PiGL (u \circ F^L_K) & \mbox{ if $K$ is a rectangle},
\end{cases}
\end{align*}
where 
$\PiT$ is defined in Lemma~\ref{lemma:hat_Pi_infty} 
and 
$\PiGL$ in Lemma~\ref{lemma:hat_Pi_1_infty}. 
Since $\PiT$ and $\PiGL$ reduce to the Gauss-Lobatto interpolation operator
on the edges of the reference element,
the operator $\Pi^L_q$ indeed maps into $S^q_0(\Omega,\Tg)$.
We recall that the element maps $F_K^L$ have the form
$$
F^L_K = F_{\KM} \circ A^L_K,
$$
where
$A^L_K:\Kh \rightarrow \Kt:= A^L_K(\Kh) =
F_{\KM}^{-1}(K)\in\Tgeneric\cup \Ttrivial$ is an affine bijection,
and 
$ F_{\KM} $ is the patch map.

Furthermore, 
$\widehat u$ denotes the pull-back of $u$ to the reference element, 
i.e.,
\begin{equation}\label{eq:hatu}
\widehat u := u|_K \circ F^L_K
\end{equation}
whereas
\begin{equation}\label{eq:tildu}
\widetilde u:= (u \circ F_{\KM})|_{\Kt} = \widehat u \circ (A^L_K)^{-1}
\end{equation}
is the corresponding function on $\Kt$.
With the patch-interpolant $\widetilde \Pi^L_q$ from (\ref{eq:tildePiq}), 
we obtain on a macro-element $\KM\in {\mathcal T}^{\M}$
\begin{equation} \label{eq:GlobMacPi}
(\Pi^L_q u)\circ F_{\KM} = \widetilde \Pi^L_q \ut.
\end{equation}
For $k \in {\mathbb N}_0$, we have for all elements $K \subset \KM$
with $\Kt = F_{\KM}^{-1}(K)$
\begin{subequations}
\label{eq:norm-equivalence}
\begin{align}
\forall v \in H^k(K)\colon\;\;
\|v \circ F_{\KM} \|_{H^k(\Kt)} &\sim \|v\|_{H^k(K)} , \\
\forall v \in W^{k,\infty}(K) \colon\;\;
\|v \circ F_{\KM} \|_{W^{k,\infty}(\Kt)} &\sim \|v\|_{W^{k,\infty}(K)},
\end{align}
\end{subequations}
where in both cases the constants implied in $\sim$ depend solely 
on $k$, the patch maps $F_{\KM}$ and the macro-element $\KM$. 

The equivalences (\ref{eq:norm-equivalence}) show that
the approximation error $v - \Pi^L_q v$ on $K$ is equivalent to the corresponding
error $\widetilde{v} - \widetilde\Pi^L_q \widetilde{v}$ on $\Kt$.
\subsection{Mesh layers and cutoff function}
\label{sec:LayCut}
For $L\in \mathbb{N}$, we subdivide the mesh $\Tg$ into
boundary layer $\lay_0$, 
transition layer $\lay_1$, and 
internal mesh elements $\lay_{\mathrm{int}}$. 
Specifically, we let
\begin{align*}
  \lay_0 &:= \left\{ K\in \Tg: \overline{K}\cap \partial\Omega \neq \emptyset \right\},\\
  \lay_1 &:= \left\{ K\in \Tg\setminus \lay_0: \exists\, J\in \lay_0 \text{ such that }
  \overline{K}\cap \overline{J} \neq \emptyset \right\},\\
  \lay_{\mathrm{int}} &:= \Tg \setminus\left( \lay_0\cup\lay_1 \right).
\end{align*}
Furthermore, we introduce the continuous, 
piecewise linear, cutoff function $\cutoff:\Omega \to [0,1]$ satisfying
\begin{equation}
  \label{eq:cutoff-prop}
\cutoff \in \mathcal{S}^1_0(\Omega, \Tg),
\qquad
\cutoff \equiv 0\text{ on all }K\in \lay_0,
\qquad
\cutoff \equiv 1\text{ on all }K\in \lay_{\mathrm{int}}.
\end{equation}
Finally,  the subdomain comprising the union of all mesh elements touching the boundary is
\begin{equation}
  \label{eq:SLzero}
\Omega^L_0 = \bigcup_{K\in \lay_0} \overline{K}\;.
\end{equation}
%
\subsection{Exponential convergence of the $hp$ approximation}
\label{sec:ExpCnvhp}
We aim to construct an approximation $v\in \VLq$, with $\VLq$ as
defined in \eqref{eq:S^q_0}, to the weak solution $u$ of the 
fractional PDE  \eqref{eq:FracLap}
that converges exponentially in the $\tH^s(\Omega)$-norm. 
By Proposition \ref{prop:LocNrm}, we fix 
$\beta \in [0, 1-s)$ and apply the triangle inequality 
to obtain
\begin{equation}
    \label{eq:error-decomp}
    \begin{aligned}
      \inf_{v\in \VLq}\| u - v \|_{\tH^s(\Omega)} &\leq \inf_{v\in \VLq}
\| \cutoff u - v \|_{\tH^s(\Omega)}  + \| (1-\cutoff) u \|_{\tH^s(\Omega)}
\\
&\leq C_{\beta, s}
\| \cutoff (u - \Pi^L_{q-1}) u\|_{H^1_\beta(\Omega)}  + \| (1-\cutoff) u \|_{\tH^s(\Omega)},
    \end{aligned}
\end{equation}
where we have used $\cutoff\in \mathcal{S}^1_0(\Omega, \Tg)$ 
so that $\cutoff \Pi^L_{q-1} u \in \VLq$ for $q \geq 2$.
In the next section, 
we estimate the second term in the right-hand side of the above inequality.
Then, in the following sections, we proceed with an estimate of the first term in the
right-hand side of \eqref{eq:error-decomp}.
We will consider separately the
reference vertex (Sec.~\ref{sec:R-c-patch}), edge (Sec.~\ref{sec:R-e-patch}),
and vertex-edge (Sec.~\ref{sec:R-ce-patch}) patches.
Finally, in Section~\ref{sec:GlobErrBd} we bring all estimates together in 
$\Omega$.
\subsubsection{Estimate of the term $(1-\cutoff) u$}
\label{sec:EstuCutoff}
The following statement is an estimate of the $\tH^s(\Omega)$-norm of the term
$u-\cutoff u$. 
\begin{lemma}
  \label{lemma:u-gu}
  Let $u$ be the solution to \eqref{eq:weakform} for $s\in (0,1)$.
  Let $L\in \mathbb{N}$ and $\cutoff$ be
defined as in \eqref{eq:cutoff-prop}. 
Then, there exist $C, b>0$ independent of $L$ such that 
\begin{equation}
  \label{eq:u-gu}
  \| u - \cutoff u \|_{\tH^s(\Omega)} \leq C\exp(-bL).
\end{equation}
\end{lemma}
\begin{proof}
We fix $\beta \in [0,1)$ additionally satisfying $\beta\in (1/2-s, 1-s)$ 
and estimate the $H^1_\beta(\Omega)$-norm of $u-\cutoff u$.
From Lemma \ref{lemma:1-cutoff} it follows that there exist constants $c,C>0$ independent
of $L$ such that
\begin{equation*}
  \| (1-\cutoff) u \|_{H^1_\beta(\Omega)} \leq C \| u \|_{H^1_\beta(S_{c\sigma ^L})}.
\end{equation*}
We now decompose $S_{c\sigma^L} $ into its components belonging to vertex, edge,
vertex-edge, and internal neighborhoods:
\begin{equation*}
  S_{c\sigma^L} = \bigcup_{\mathbf{v}\in \mathcal{V}} \left((\omegac \cap S_{c\sigma^L})\cup \bigcup_{\mathbf{e}\in \mathcal{E_\mathbf{v}}} (\omegace \cap S_{c\sigma^L}) \right) \cup \bigcup_{\mathbf{e}\in \mathcal{E}}(\omegae\cap S_{c\sigma^L}) \cup (\Omega_\mathrm{int} \cap S_{c\sigma^L}).
\end{equation*}
We start with vertex neighborhoods $\omegac$:
Since $\beta > 1/2-s$, we may choose $\varepsilon$ sufficiently small such that
$\beta-1/2+s-\varepsilon > 0$. 
For any $\mathbf{v}\in \mathcal{V}$, we obtain using the weighted regularity estimate \eqref{eq:analytic-u-c-all} for $p=0,1$
\begin{align*}
  \|u \|_{H^1_\beta(\omegac\cap S_{c\sigma ^L})} 
  &
      \lesssim \|r_{\mathbf{v}}^{-1/2-s+\epsilon+\beta-1+1/2+s-\epsilon} u\|_{L^2(\omegac\cap S_{c\sigma^L})} \\
  & \qquad
    + \|r_{\mathbf{v}}^{1/2-s+\epsilon+\beta-1/2+s-\epsilon} \nabla u \|_{L^2(\omegac\cap S_{c\sigma ^L})} 
    \\
  &\leq (c\sigma^L)^{\beta-1/2+s-\epsilon}\| u \|_{H^1_{1/2-s+\epsilon}(\omegac\cap S_{c\sigma ^L})} \\
  &\stackrel{\eqref{eq:analytic-u-c-all}}{\lesssim} (c\sigma^L)^{\beta-1/2+s-\epsilon}.
\end{align*}
We next estimate the $H^1_\beta$-norm of the interpolation error 
on edge neighborhoods $\omegae$: 
for any $\mathbf{e}\in \mathcal{E}$, we use the weighted regularity \eqref{eq:analytic-u-e-all} 
with $\ppar = 0$ and $\pperp = 0,1$ to bound
\begin{align*}
 & \|u \|_{H^1_{\beta}(\omegae\cap S_{c\sigma ^L})}
  \\ & \qquad \simeq
  \|r_{\mathbf{e}}^{\beta-1} u \|_{L^2(\omegae\cap S_{c\sigma ^L})} +
  \|r_{\mathbf{e}}^\beta \Dpar u \|_{L^2(\omegae\cap S_{c\sigma ^L})} +
  \|r_{\mathbf{e}}^\beta \Dperp u \|_{L^2(\omegae\cap S_{c\sigma ^L})} \\
   & \qquad \leq
    \|r_{\mathbf{e}}^{-1/2-s+\epsilon+\beta-1+1/2+s-\epsilon} u \|_{L^2(\omegae\cap S_{c\sigma ^L})}
     \\ & \qquad \qquad +
     \|r_{\mathbf{e}}^{-1/2-s+\epsilon+\beta+1/2+s-\epsilon} \Dpar u \|_{L^2(\omegae\cap S_{c\sigma ^L})}
     \\ & \qquad \qquad +
   \|r_{\mathbf{e}}^{1/2-s+\epsilon+\beta-1/2+s-\epsilon} \Dperp u \|_{L^2(\omegae\cap S_{c\sigma ^L})} \\
  &\qquad \leq
  (c\sigma^L)^{\beta+1/2+s-\epsilon}\|r_{\mathbf{e}}^{-1/2-s+\epsilon} \Dpar u \|_{L^2(\omegae\cap S_{c\sigma ^L})} 
    \\ & \qquad \qquad
         + (c\sigma^L)^{\beta-1/2+s-\epsilon}\bigg(\|r_{\mathbf{e}}^{-1/2-s+\epsilon} u \|_{L^2(\omegae\cap S_{c\sigma ^L})}
  \\ & \qquad \qquad
       + \|r_{\mathbf{e}}^{1/2-s+\epsilon} \Dperp u \|_{L^2(\omegae\cap S_{c\sigma ^L})} \bigg)\\
  &\qquad \stackrel{\eqref{eq:analytic-u-e-all}}{\lesssim}  (c\sigma^L)^{\beta-1/2+s-\epsilon}.
\end{align*}
The error on the vertex-edge neighborhood $\omegace$ can be bounded 
for any $\mathbf{v} \in \mathcal{V}$ and any $\mathbf{e}\in \mathcal{E}_{\mathbf{v}}$ 
using $r_{\mathbf{v}}(x)\gtrsim r_{\mathbf{e}}(x)$ for all $x\in \omegace$
as well as the weighted regularity \eqref{eq:analytic-u-ce-all} with $\ppar,\pperp$ satisfying $\ppar + \pperp \leq 1$
\begin{align*}
  & \|u \|_{H^1_\beta(\omegace\cap S_{c\sigma ^L})}
  \\ & \qquad \simeq
  \|r_{\mathbf{e}}^{\beta-1} u \|_{L^2(\omegace\cap S_{c\sigma ^L})} +
  \|r_{\mathbf{e}}^\beta \Dpar u \|_{L^2(\omegace\cap S_{c\sigma ^L})} +
  \|r_{\mathbf{e}}^\beta \Dperp u \|_{L^2(\omegace\cap S_{c\sigma ^L})} 
  \\ & \qquad
      \leq
\|r_{\mathbf{e}}^{-1/2-s+\epsilon+\beta-1+s+1/2-\epsilon}r_{\mathbf{v}}^{\epsilon-\epsilon} u \|_{L^2(\omegace\cap S_{c\sigma ^L})} \\
   & \qquad \qquad 
       + \|r_{\mathbf{e}}^{-1/2-s+\epsilon+\beta+s+1/2-\epsilon}r_{\mathbf{v}}^{1+\epsilon-1-\epsilon} \Dpar u \|_{L^2(\omegace\cap S_{c\sigma ^L})} \\
   & \qquad \qquad + \|r_{\mathbf{e}}^{1/2-s+\epsilon+\beta-1/2+s-\epsilon} r_{\mathbf{v}}^{\epsilon-\epsilon}\Dperp u \|_{L^2(\omegace\cap S_{c\sigma ^L})} 
     \\
  &\qquad \stackrel{\eqref{eq:analytic-u-ce-all}}{\lesssim} (c\sigma^L)^{\beta-1/2+s-2\epsilon},
\end{align*}
where we assumed $\varepsilon$ to be chosen small enough such that $\beta-1/2+s-2\epsilon > 0$.
Finally, as $c,\sigma$ are fixed, we may assume that 
$\Omega_{\mathrm{int}}\cap S_{c\sigma^L} = \emptyset$ 
by replacing $L$ by $L+L_0$ with a fixed $L_0 \in \mathbb{N}$ 
large enough and independent of $L$, 
which only changes the constant $b$ in the exponential estimate.
We have thus obtained that  
\begin{equation*}
  \| (1-\cutoff) u \|_{H^1_\beta(\Omega)} \leq  \| u \|_{H^1_\beta(S_{c\sigma ^L})} \leq C' \exp(-bL).
\end{equation*}
Applying Proposition \ref{prop:LocNrm} concludes the proof.
\qed
\end{proof}
\begin{figure}
  \centering
  \begin{subfigure}[b]{.3\textwidth}
    \centering
   \includegraphics[width=.9\textwidth]{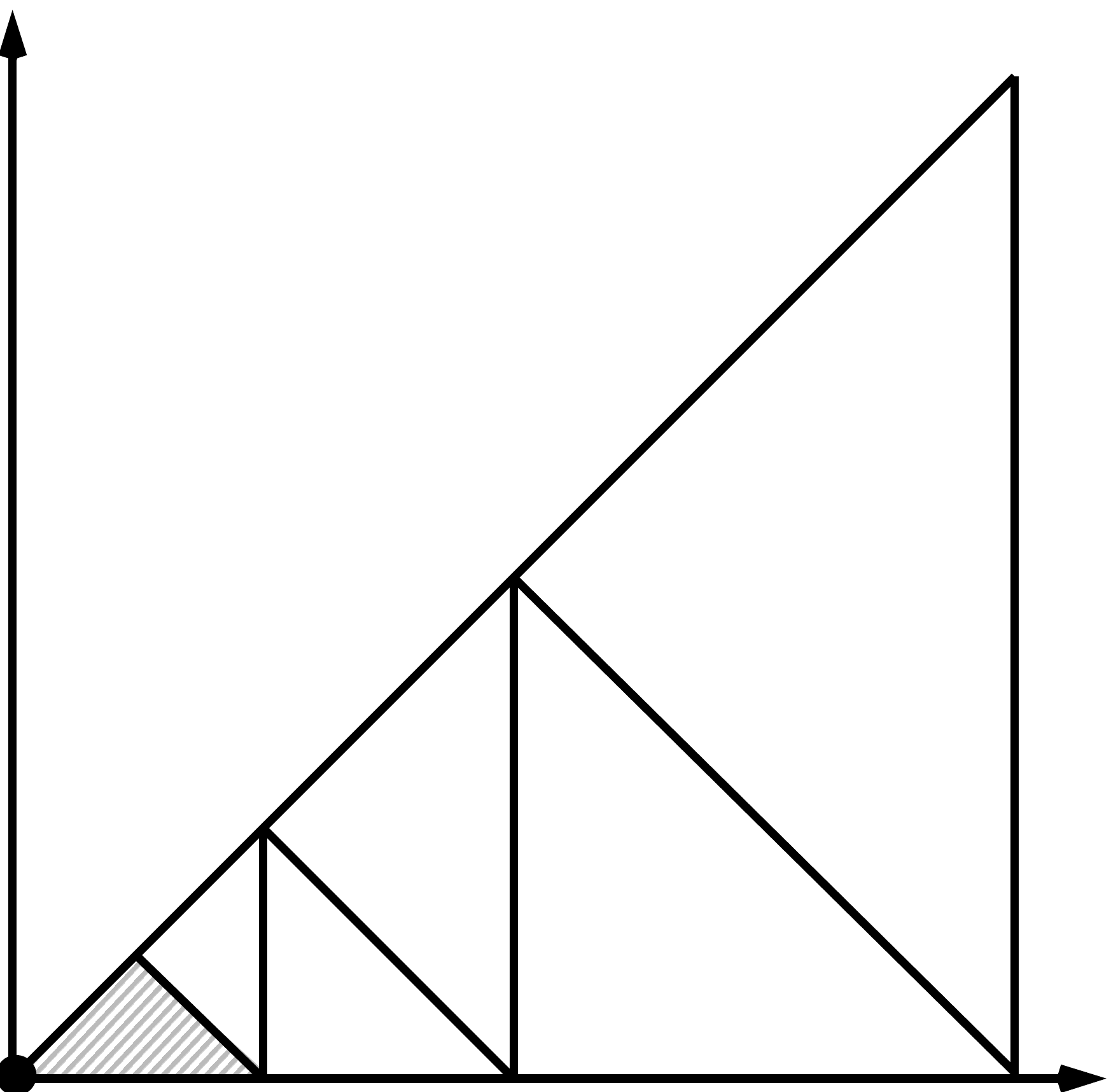}
   \caption{}\label{subfig:L0-corner}
  \end{subfigure}
  \begin{subfigure}[b]{.3\textwidth}
    \centering
   \includegraphics[width=.9\textwidth]{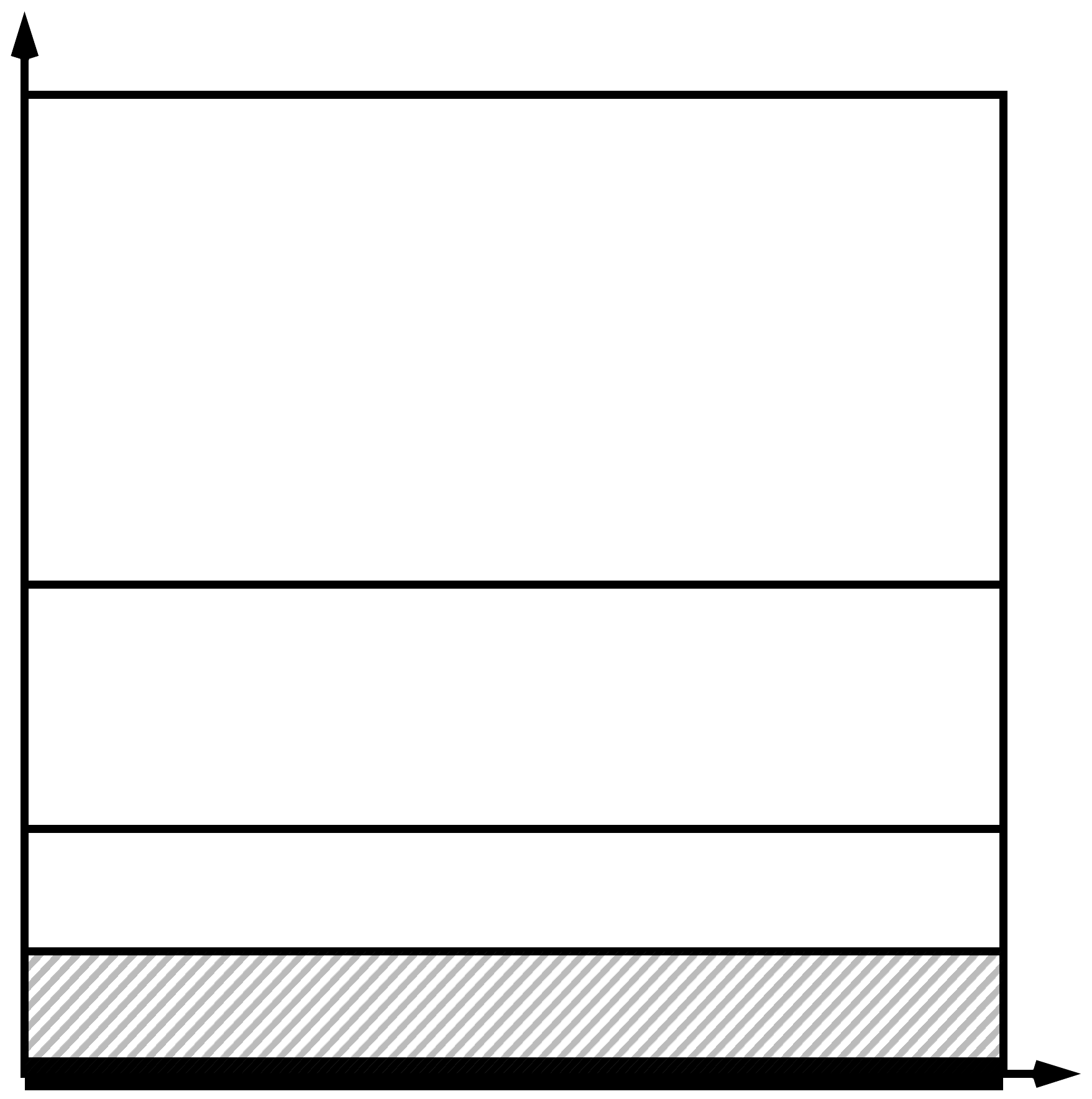}
   \caption{}\label{subfig:L0-edge}
  \end{subfigure}
\begin{subfigure}[b]{.3\textwidth}
    \centering
    \includegraphics[width=.9\textwidth]{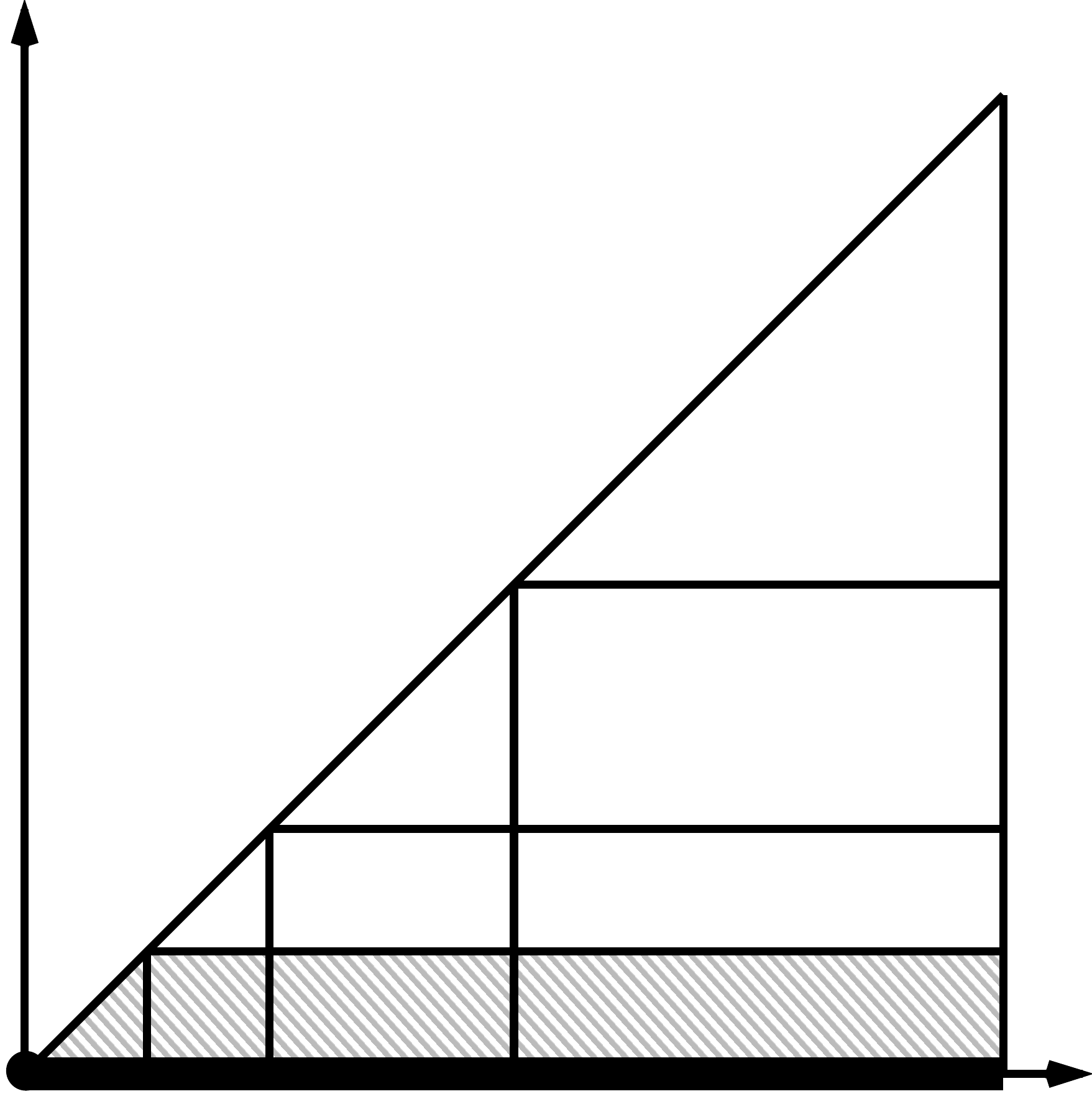}
   \caption{}\label{subfig:L0-ce}
  \end{subfigure}
  \caption{Boundary elements (displayed shaded)
    $\tlayC_0$,
    $\tlayE_0$, and
    $\tlayCE_0$
    in the
    (\subref{subfig:L0-corner}) vertex,
    (\subref{subfig:L0-edge}) edge,
    and
    (\subref{subfig:L0-ce}) vertex-edge reference patches.
  }
  \label{figL0-refs}
\end{figure}
\subsubsection{$hp$-FE approximation in reference vertex patch 
               $\Tcor$}
\label{sec:R-c-patch}
%
%
We denote $\mathbf{v}= (0,0)$ and
$ \trv = \dist(\mathbf{v}, \cdot)$. 
Furthermore, let 
\begin{equation*}
  \tlayC_0 = \{K\in \Tcor : \overline{K}\cap \mathbf{v}\neq \emptyset\}, \qquad
 \tTCint = \hT\setminus \bigcup_{K\in \tlayC_0} \overline{K},
\end{equation*}
be, respectively, the elements abutting the singular vertex and the interior part of the vertex reference patch, 
see Fig.~\ref{subfig:L0-corner}.
\begin{lemma}[$hp$-FE approximation in reference vertex patch $\Tcor$]
\label{lemma:R-corner-patch}
For fixed $s\in(0,1)$ and $\gamma > 0$, 
let $\tu$ satisfy the following: for all $\epsilon>0$ there exist a constant $C_\epsilon>0$ such that 
for all $\alpha\in \bbN^2_0$ it holds, with $\alpham=p$, that
\begin{equation}\label{eq:R-analytic-u-c}
\norm{\tilde{r}_{\mathbf{v}}^{p-1/2-s+\varepsilon} \partial^\alpha \tu }_{L^2(\hT)}
\leq 
C_{\varepsilon} \gamma^{p+1}p^p.
\end{equation}
Then,  for all $\beta > 1/2-s$ and all $0 < \epsilon < \beta + s - 1/2$,
there exist constants $b_{\Co} > 0$ (depending only on $\gamma$, $\beta$, $s$,
$\sigma$) and $C_{\Co}>0$ (depending additionally on $\epsilon$)
such that for every $L$, $q\in \bbN$
\begin{equation}\label{eq:ExCnvC}
\| \trv^{\beta-1}(\tu - \tPi^L_{q}\tu )\|_{L^2(\tTCint)}
+
\| \trv^{\beta}\nabla(\tu - \tPi^L_{q}\tu )\|_{L^2(\tTCint)}
\leq 
C_{\Co} C_\varepsilon\exp(-b_{\Co} q) .
\end{equation}
\end{lemma}
\begin{proof}
  All elements $\Kt \in \Tcor$ are shape regular: we denote by $h_{\Kt}$ their diameter. 
  For all $\Kt\in \Tcor \setminus \tlayC_0$, we have
  $\trv|_{\Kt}\simeq h_{\Kt}$ with equivalence constant uniform over
  $\Tcor\setminus \tlayC_0$.
   From this equivalence and 
    \eqref{eq:R-analytic-u-c} it follows that, 
    for all $\Kt\in \Tcor\setminus \tlayC_0$ and all $\alpha\in \N^2_0$, 
    there exists a constant $C_1 > 0$ such that

    \begin{equation*}
      \norm{\trv^{\alpham}\partial^\alpha \tu }_{L^2(\Kt)}
      \leq C_1C_\varepsilon h_{\Kt}^{1/2+s-\varepsilon} \gamma^{\alpham+1}\alpham^{\alpham}.
    \end{equation*}

    By a
    scaling argument, then, there exists a constant $\gamma_1 > 0$ such that 
    for all $\alpha\in \N^2_0$ and all $\Kt\in \Tcor\setminus \tlayC_0$,
    \begin{equation*}
      \norm{\hpartial^\alpha \left(\tu \circ A^L_{\Kt}\right) }_{L^2(\Kh)}
      \leq 
      C_1 C_\varepsilon h_{\Kt}^{-1/2+s-\varepsilon}\gamma_1^{\alpham+1} \gamma^{\alpham+1}\alpham^{\alpham}
    \end{equation*}
    with $\Kh = \hT = (A^L_{{\Kt}})^{-1}(\Kt)$. 
    Recalling $\hu = \tu \circ A^L_{\Kt}$, we
    can now exploit the embedding of $H^2(\Kh)$ into $L^\infty(\Kh)$ to obtain
    the existence of constants $C_2, \gamma_2>0$ such that
    \begin{equation*}
      \forall \alpha\in \N^2_0 \colon \quad
      \| \hpartial^\alpha \hu \|_{L^\infty(\Kh)}
      \leq C_2 h_{\Kt}^{s-1/2-\epsilon}\gamma_2^{\alpham+3} (\alpham+2)^{\alpham+2}.
    \end{equation*}
    It follows that there exists
    $C_3$, $\gamma_3>0$ such that
    \begin{equation}
      \label{eq:Linfty-Khat-c}
\forall \Kt\in \Tcor\setminus \tlayC_0\quad  \forall \alpha\in \bbN^2_0 \colon
\quad 
      \|\hpartial^\alpha \hu\|_{L^\infty(\Kh)}
      \leq 
      C_3 h_{\Kt}^{s-1/2-\epsilon} \gamma_3^{\alpham+1}\alpham^{\alpham}. 
    \end{equation}
  From Lemma \ref{lemma:hat_Pi_infty} and a scaling argument,
  it then follows that, for all $L, q\in\bbN$,
  \begin{equation*}
\| \trv^{\beta-1}(\tu - \tPi^L_{q}\tu )\|^2_{L^2(\Kt)}
+
\| \trv^{\beta}\nabla(\tu - \tPi^L_{q}\tu )\|^2_{L^2(\Kt)}
\lesssim h_{\Kt}^{2\beta+{2}s-1-2\epsilon} e^{-2bq}.
  \end{equation*}
Since $\beta>1/2-s$, the power of $h_{\Kt}$ is non-negative for every $\epsilon<\beta+s-1/2$ 
and 
summing the bound over all elements $\Kt\in \Tcor \setminus \tlayC_0$ concludes the proof
by a geometric series argument. 
\qed
\end{proof}
\subsubsection{$hp$-FE approximation in the 
               reference edge patch $\Tedg$}
\label{sec:R-e-patch}
%
In this section, we denote $\mathbf{e}=\{y = 0 \}$ and
 $\tre = \dist(\mathbf{e}, \cdot)$.
 Let $\tD_{x_\parallel} = \partial_x$  and $\tD_{x_\perp} = \partial_y$.
Furthermore, let 
\begin{equation*}
  \tlayE_0 = \{K\in \Tedg : \overline{K}\cap \mathbf{e}\neq \emptyset\}, \qquad 
 \tSEint = \widehat{S}\setminus \bigcup_{K\in \tlayE_0} \overline{K}
\end{equation*}
be, respectively, the elements abutting the singular boundary and the interior part of the edge reference patch,
see Fig.~\ref{subfig:L0-edge}.

\begin{lemma}[$hp$-FE approximation in reference edge patch $\Tedg$]
\label{lemma:R-edge-patch}
Let $s\in(0,1)$ and $\gamma > 0$ be fixed, and let $\tu$ be such that for all $\epsilon>0$ there exists $C_\varepsilon > 0$ such htat 
\begin{equation}
\label{eq:analytic-u-e-R}
\forall (\pperp , \ppar)\in \bbN_0^2 \colon \quad 
 \norm{\tilde{r}_{\mathbf{e}}^{\pperp-1/2-s+\varepsilon} \tD^{\pperp}_{x_\perp} \tD^{\ppar}_{x_{\parallel}} \tilde{u} 
      }_{L^2(\hS)}
\leq 
C_{\varepsilon}\gamma^{p+1}p^p ,
\end{equation}
with $p = \pperp+\ppar$.
Then, for all $\beta>1/2-s$ and all $0 < \epsilon < \beta + s -1/2$, 
there
exist constants $b_\Edg > 0$ (depending only on $\gamma$, $\beta$, $s$,
$\sigma$) and $C_{\Edg}>0$ (depending additionally on $\epsilon$)
such that for every $L$, $q\in \bbN$
\begin{equation}\label{eq:ExCnvE}
\| \tre^{\beta-1}(\tilde{u} - \tPi^L_q\tilde{u}) \|_{L^2(\tSEint)}
+
\| \tre^{\beta}\nabla(\tilde{u} - \tPi^L_q\tilde{u}) \|_{L^2(\tSEint)}
\leq C_{\Edg}  
C_\varepsilon
\exp(-b_{\Edg} q).
\end{equation}
\end{lemma}
%
\begin{proof}
  We denote by $h_{\parallel, \Kt}$ and $h_{\perp, \Kt}$ the edge-lengths of the rectangle
  $\Kt\in \Tedg$ in, respectively, parallel and perpendicular directions to $\mathbf{e}$.
For all $\Kt\in \Tedg \setminus \tlayE_0$, we have
  $\tre|_{\Kt}\simeq h_{\perp, \Kt}$ with equivalence constant uniform over
  $\Tedg\setminus \tlayE_0$. From \eqref{eq:analytic-u-e-R}, an anisotropic 
  scaling argument, and a Sobolev embedding it follows that
  there exist $\widetilde{C}, \widetilde{\gamma}>0$ such that
  \begin{equation}
    \label{eq:edge-reg-1}
    \forall \Kt\in \Tedg\setminus\tlayE_0,\, \forall (\pperp, \ppar)\in \bbN_0^2: \qquad
    \| \hpartial^{(\ppar, \pperp)}\hu \|_{L^{\infty}(\Kh)}\leq \widetilde{C} h_{\perp,\Kt}^{s-\epsilon} \widetilde{\gamma}^{p+1}p^{p}
  \end{equation}
with 
$\Kh = \hS = (A^L_{{\Kt}})^{-1}(\Kt)$ 
and 
$p = \ppar+\pperp$ (see the derivation of \eqref{eq:Linfty-Khat-c} for the detailed steps).
  From Lemma~\ref{lemma:hat_Pi_1_infty} and a scaling argument, it then follows that
  \begin{equation*}
\| \tre^{\beta-1}(\tilde{u} - \tPi^L_q\tilde{u}) \|^2_{L^2(\Kt)}
+
\| \tre^{\beta}\nabla(\tilde{u} - \tPi^L_q\tilde{u}) \|^2_{L^2(\Kt)}
    \lesssim  
    h_{\perp, \Kt}^{2\beta+2s-1-2\epsilon} e^{-2bq}.
  \end{equation*}
  Summing this bound over all $\Kt\in \Tedg\setminus \tlayE_0$ using a geometric series argument concludes the proof
since $\beta+ s-1/2 -\varepsilon > 0$.
  \qed
\end{proof}

\subsubsection{$hp$-FE approximation in the reference vertex-edge patch $\Tce$}
\label{sec:R-ce-patch}
In this section, 
we denote $\mathbf{v} = (0,0)$, 
$\mathbf{e}=\{y = 0 \}$, $\trv = \dist(\mathbf{v}, \cdot)$, and 
$\tre = \dist(\mathbf{e}, \cdot)$.
Let $\tD_{x_\parallel} = \partial_x$  and $\tD_{x_\perp} = \partial_y$.
Furthermore, let 
\begin{equation*}
  \tlayCE_0 = \{K\in \Tce : \overline{K}\cap (\mathbf{e}\cup \mathbf{v})\neq \emptyset\}, \qquad 
\tTCEint= \widehat{T}\setminus \bigcup_{K\in \tlayCE_0} \overline{K} 
\end{equation*}
be, respectively, the elements abutting the singular boundary 
and the interior part of the vertex-edge reference patch, see Fig.~\ref{subfig:L0-ce}.
\begin{lemma}[$hp$-FE approximation in reference vertex-edge patch $\Tce$]
\label{lemma:R-ce-patch}
Let $s\in(0,1)$ and $\gamma > 0$ be fixed, and let $\tu$  be such that for all $\epsilon>0$ and
there exists $C(\varepsilon) > 0$ such that 
%
for all $(\ppar, \pperp )\in \bbN_0^2$ with $\ppar+\pperp=p$
\begin{equation}
\label{eq:R-analytic-u-ce}
\norm{\tilde{r}_{\mathbf{e}}^{\pperp-1/2-s + \varepsilon} \trv ^{\ppar+\varepsilon} \tD^{\pperp}_{x_\perp} 
        \tD^{\ppar}_{x_\parallel} u}_{L^2(\hT)} 
\leq 
C_{\varepsilon} \gamma^{p+1} p^p.
\end{equation}
Then, for all $\beta>1/2-s$ and all $0 < \epsilon < \beta/2 + s/2 - 1/4$,
there exist constants $b_\CE > 0$ (depending only on $\gamma$, $\beta$, $s$,
$\sigma$) and $C_{\CE}>0$ (depending additionally on $\epsilon$) 
such that for every $L$, $q\in \bbN$
\begin{equation}\label{eq:ExCnvM}
\| \tre^{\beta-1}(\tilde{u} - \tPi^L_q\tilde{u}) \|_{L^2( \tTCEint)}
+
\| \tre^\beta\nabla(\tilde{u} - \tPi^L_q\tilde{u}) \|_{L^2(\tTCEint )}
\leq
C_{\CE} {C_\varepsilon }\exp(-b_{\CE} q).
\end{equation}
\end{lemma}
\begin{proof}
  Let $\Kt$ be an element not belonging to $\tlayCE_0$.
  We denote by $h_{\parallel, \Kt}$ and $h_{\perp, \Kt}$ the size of the rectangle $\Kt$
  in, respectively, parallel and perpendicular directions to $\mathbf{e}$.
  We have $\tre|_{\Kt}\simeq h_{\perp, \Kt}$ and $\trv|_{\Kt}\simeq
  h_{\parallel, \Kt}$ with uniform equivalence constants. From \eqref{eq:R-analytic-u-ce}, a
  scaling argument, and a Sobolev imbedding, it follows that there exist
  $\widetilde{C}, \widetilde{\gamma}>0$ such that
    for all $(\pperp, \ppar)\in \bbN_0^2$ with $p = \ppar + \pperp$
  \begin{equation}
    \label{eq:mix-reg-1}
\forall \Kt\in \Tce\setminus \tlayCE_0 \colon \quad 
    \| \hD^{\ppar}_{x_\parallel} \hD^{\pperp}_{x_\perp}\hu \|_{L^\infty(\Kh)}\leq \widetilde{C} h_{\perp,\Kt}^{s-\epsilon} h_{\parallel, \Kt}^{-1/2-\epsilon}\widetilde{\gamma}^{p+1}p^p. 
  \end{equation}
By a scaling argument (dropping temporarily the subscript $\cdot_{,\Kt}$)
\begin{align*}
&\| \tre^{\beta-1}(\tilde{u} - \tPi^L_q\tilde{u}) \|^2_{L^2(\Kt)}
+
\| \tre^\beta\nabla(\tilde{u} - \tPi^L_q\tilde{u}) \|^2_{L^2(\Kt)}
  \\
  & \qquad \lesssim
    \begin{multlined}[t][.9\textwidth]
    h_{\perp}^{2\beta}
      \bigg(   \frac{h_\parallel}{h_\perp}  \left( \| \hu -
        \hPi_q \hu \|^2_{L^2(\Kh)} + \| \hD_{\perp} (\hu-\hPi_q\hu)\|^2_{L^2(\Kh)}
      \right) 
      \\ +  \frac{h_{\perp} }{h_{\parallel}} \| \hD_{\parallel} (\hu-\hPi_q\hu)\|^2_{L^2(\Kh)}\bigg)
    \end{multlined}\\
  & \qquad \leq h_\perp^{2\beta-1}h_{\parallel} \| \hu - \hPi_q  \hu \|^2_{W^{1,\infty}(\Kh)},
    \end{align*}
    where the penultimate estimate follows from 
    $h_\perp\simeq\tre \lesssim \trv\simeq h_\parallel$ in $\Tce \backslash \tlayCE_0$.
From Lemmas \ref{lemma:hat_Pi_infty} and \ref{lemma:hat_Pi_1_infty},  using
\eqref{eq:mix-reg-1} then gives
\begin{equation}
  \label{eq:last-Kt-ve}
\| \tre^{\beta-1}(\tilde{u} - \tPi^L_q\tilde{u}) \|^2_{L^2(\Kt)}
+
\| \tre^\beta\nabla(\tilde{u} - \tPi^L_q\tilde{u}) \|^2_{L^2(\Kt)}
  \lesssim 
h_{\perp, \Kt}^{2\beta+2s-1-2\epsilon} h_{ \parallel , \Kt}^{-2\epsilon}e^{-2bq}.
\end{equation}
From $\beta>1/2-s$ it follows that $\varepsilon>0$ can be chosen so that
$2\beta +2s-1 > 4\varepsilon$. 
In addition, 
$h_{\perp,\Kt}\leq h_{\parallel, \Kt}$. 
Hence, there exists $\delta>0$ such that, 
for all $\epsilon$ as specified above, 
$h_{\perp, \Kt}^{2\beta+2s-1-2\epsilon} h_{\parallel , \Kt}^{-2\epsilon} \leq h_{\perp, \Kt}^\delta$. 
Then,
\begin{align*}
        \sum_{\Kt \in \Tce\setminus \tlayCE_0} h_{\perp,\Kt}^{ \delta } 
          &
         =(1-\sigma)^\delta\sum_{i=0}^{L-1} \sum_{j=i}^{L-1}\sigma^{\delta j}
       =\frac{(1-\sigma)^\delta}{1-\sigma^\delta}\sum_{i=0}^{L-1} \left( \sigma^{\delta i}  - \sigma^{\delta L}\right) 
  \\          & \leq 
\frac{(1-\sigma)^\delta}{(1-\sigma^\delta)^2}.
\end{align*}
Summing \eqref{eq:last-Kt-ve} over all elements in $\Tce\setminus \tlayCE_0$ concludes the proof.
\qed
\end{proof}
  \begin{remark}
  The dependence on $\epsilon$ of the constants $C_{\Co}$, $C_{\Edg}$, $C_{\CE}$ of Lemmas \ref{lemma:R-corner-patch},
  \ref{lemma:R-edge-patch}, and \ref{lemma:R-ce-patch} can be dropped if, for $\bullet\in \{ \Co,\Edg,\CE \}$,
  the constant $C_{\bullet}$ is replaced
by $\tC_{\Co} L$ in \eqref{eq:ExCnvC}, by $\tC_{\Edg}L$  in \eqref{eq:ExCnvE}, and by $\tC_{\CE} L^2$ in
  \eqref{eq:ExCnvM}. The newly introduced constants $\tC_\bullet$ are independent of the
  choice of $\epsilon$.

  This has no effect on the final result. Considering the
  dependence of $C_\bullet$ on $\epsilon$, a fixed value of $\epsilon$ is chosen
  in the proof of Theorem \ref{thm:hpExpConv}, independently of $\beta$ and $s$. If one
  were to use instead the results with the constants $\tC_\bullet$, the terms in $L$ and $L^2$
  can be absorbed in the exponential $e^{-b_\bullet q}$ after having set $q\sim L$.
  \eremk
  \end{remark}
\subsubsection{Global error bound (Proof of Theorem~\ref{thm:hpExpConv})}
\label{sec:GlobErrBd}
\begin{proof}[Proof of Theorem \ref{thm:hpExpConv}]
Recall that $W^L_q = \mathcal{S}^q_0(\Omega, \Tg) $ is the space of continuous,
piecewise polynomials of maximum degree $q$ on a mesh with $L$ levels of refinement.
From \eqref{eq:error-decomp}, Lemma \ref{lemma:u-gu}, and Lemma
  \ref{lemma:cutoff-equiv} it follows that for all $\beta\in [0, 1-s)$
  \begin{equation}
    \label{eq:u-v_triangle_proof}
    \inf_{v\in W^L_q}\| u - v \|_{\tH^s(\Omega)} \lesssim \| u - \Pi^L_{q-1}u \|_{H^1_\beta(\Omega\setminus \Omega^L_0)} + \exp(-b_1L).
  \end{equation}
Remark that we can choose (potentially overlapping) $\omegac$, $\omegae$, and
$\omegace$ so that for all $\KM\in \cT^\M_{\Edg}$, $\KM\subset\omegae$; for
all $\KM\in \cT^\M_{\CE}$, $\KM\subset \omegace$; for all 
$\KM\in \cT^\M_{\Co}$, either $\KM\subset \omegac$ or $\KM\subset \omegae$.
In other words, the edge and vertex-edge patches in the domain $\Omega$ are,
respectively, contained in $\omegae$ and $\omegace$; the vertex patch is either
contained in $\omegac$ or $\omegae$, with origin mapped to a point 
on a vertex or along an edge.

Suppose now that $u$ satisfies
\eqref{eq:analytic-u-c-all}--\eqref{eq:analytic-u-int}. Consider a patch $\KM\in
\cT^\M$ and denote $\tu = u \circ F_{\KM}$.
Let $\partial_{\tx}$ be
differentiation with respect to the variable $\tx  = F_{\KM}^{-1}(x)$ and
let $D_{\tx_\perp}$ and $D_{\tx_\parallel}$ be differentiation in directions respectively
perpendicular and parallel to an edge pulled back to the reference patch.
Let also $\tbv = (0 , 0)$, $\tbe = (0, 1)\times \{0\}$ and denote
$\trv(\tx) = |\tx - \tbv|$, and $\tre(\tx) = \dist(\tx, \tbe)$, for all $\tx\in F_{\KM}^{-1}(\KM)$.

\paragraph{\textbf{Case $\KM \in \cT^\M_{\Edg}$.}}
Since $F_{\KM}$ is affine and since it maps the closure of $\{(x_1, x_2) \in \hS: x_2=0\}$ to
$\partial \Omega\cap \partial \KM$, its Jacobian $J_{F_{\KM}}$ can be
written as the composition of an upper triangular matrix $U_{\KM}$ and a rotation $R_{\KM}$:
\begin{equation*}
  J_{F_{\KM}} =  R_{\KM} U_{\KM} .
\end{equation*}
Without loss of generality, 
    we may assume the coordinate systems oriented such that
    the vector $(1, 0)^\top$, parallel to the singular edge in $\hS$,
is mapped to $\epar = R_{\KM}(1,0)^\top$. 
We remark that, since $U_{\KM}$ is upper triangular, there exists $\eta_{\KM}\in \R$ such that
\begin{equation*}
  U_{\KM} (1, 0)^\top = \eta_{\KM} (1, 0)^\top.
\end{equation*}
Hence,
\begin{equation*}
  D_{\tx_\parallel} =
  \begin{pmatrix}
    1 \\ 0
  \end{pmatrix}
  \cdot
  \nabla_{\tx} =
  \begin{pmatrix}
    1 \\ 0
  \end{pmatrix}
  \cdot (J_{F_{\KM}}^\top\nabla_x) = \eta_{\KM} 
  \begin{pmatrix}
    1 \\ 0
  \end{pmatrix} \cdot (R_{\KM}^\top \nabla_x) = \eta_{\KM} \epar \cdot \nabla_x = \eta_{\KM} D_{x_\parallel}.
\end{equation*}
By a similar argument, there exist $\beta_1, \beta_2 \in \R$ such that
\begin{equation*}
  D_{\tx_\perp} = \beta_1 D_{x_\parallel} + \beta_2 D_{x_\perp}.
\end{equation*}
Finally, there exists $c_{K_M} > 0 $ such that for all  $x\in \hS$,
\begin{equation*}
  \frac{1}{c_{\KM}} \tre(x) \leq r_{\mathbf{e}}(F_{\KM}(x)) \leq c_{\KM} \tre(x).
\end{equation*}
Hence,
\begin{align*}
  & \norm{\tilde{r}_{\mathbf{e}}^{\pperp-1/2-s+\varepsilon} D^{\pperp}_{\tx_\perp} D^{\ppar}_{\tx_{\parallel}} \tilde{u} 
  }_{L^2(F_{\KM}^{-1}(\KM))}
  \\ & \qquad
  \leq 
  C c_{\KM}^{\pperp-1/2-s+\epsilon}\eta_{\KM}^{\ppar}
  \|r_{\mathbf{e}}^{\pperp-1/2-s+\epsilon} (\beta_1D_{x_\parallel} + \beta_2D_{x_\perp})^{\pperp} D_{x_\parallel}^{\ppar} u \|_{L^2(\KM)}
  \\ & \qquad
       \leq
       C c_{\KM}^{\pperp-1/2-s+\epsilon}\eta_{\KM}^{\ppar} (\beta_1+\beta_2)^{\pperp}
       \max_{j=0, \dots, \pperp}\|r_{\mathbf{e}}^{\pperp-1/2-s+\epsilon}  D_{x_\perp}^jD_{x_\parallel}^{\ppar +\pperp- j} u \|_{L^2(\KM)}.
\end{align*}
It follows then from
\eqref{eq:analytic-u-e-all} that
there exist $\tC_\epsilon, \tgamma>0$ such that, for all 
$(\pperp , \ppar)\in \bbN_0^2$ with $p=\pperp+\ppar$ and 
for all $\KM \in \cT^\M_{\Edg}$, 
we obtain
\begin{equation*}
  \norm{\tilde{r}_{\mathbf{e}}^{\pperp-1/2-s+\varepsilon} D^{\pperp}_{\tx_\perp} D^{\ppar}_{\tx_{\parallel}} \tilde{u} 
      }_{L^2(F_{\KM}^{-1}(\KM))}
\leq 
\tC_{\varepsilon}\tgamma^{p+1}p^p .
\end{equation*}
\paragraph{\textbf{Case $\KM\in \cT^\M_{\CE}$.}}
This case is treated as the previous one, noting that in addition
\begin{equation}
  \label{eq:rv-trv}
  \frac{1}{c_{\KM}} \trv(x) \leq r_{\mathbf{v}}(F_{\KM}(x)) \leq c_{\KM} \trv(x).
\end{equation}
We obtain
from
\eqref{eq:analytic-u-ce-all} that
there exist $\tC_\epsilon, \tgamma>0$ such that, for all 
$(\pperp , \ppar)\in \bbN_0^2$ with $p=\pperp+\ppar$ 
and for all $\KM\in \cT^\M_{\CE}$,
\begin{equation*}
\norm{\tilde{r}_{\mathbf{e}} ^{\pperp-1/2-s + \varepsilon} \trv ^{\ppar+\varepsilon} D^{\pperp}_{\tx_\perp} D^{\ppar}_{\tx_\parallel} \tilde u}_{L^2(F_{\KM}^{-1}(\KM))} 
\leq \tC_{\varepsilon} \tgamma^{p+1} p^p.
\end{equation*}
\paragraph{\textbf{Case $\KM\in \cT^\M_{\Co}$.}}
If $\KM\subset \omegac$, then \eqref{eq:rv-trv} holds. 
In addition, there exists a constant $\tilde{c}$ such that, 
for all $\alpha\in \N^2_0$,
\begin{equation}
  \label{eq:td-d}
  \|\trv^{\alpham-1/2-s+\epsilon}\partial_{\tx}^\alpha \tu\|_{L^2(F_{\KM}^{-1}(\KM))}
  \leq
  c_{\KM}^{\alpham-1/2-s+\epsilon}\tilde{c}^{\alpham} 
  \max_{\beta\leq \alpha} \|r_{\mathbf{v}}^{\alpham-1/2-s+\epsilon}\dbeta u\|_{L^2(\KM)}.
\end{equation}
Therefore, \eqref{eq:analytic-u-c-all} implies 
\begin{equation*}
\norm{\tilde{r}_{\mathbf{v}}^{\alpham-1/2-s+\varepsilon} \partial_{\tx}^\alpha \tilde{u} }_{L^2(F_{\KM}^{-1}(\KM))}
\leq 
\tC_{\varepsilon} \tgamma^{\alpham+1}\alpham^{\alpham}.
\end{equation*}
If instead $\KM\subset \omegae$, there
exists $c_{\KM}$ such that
\begin{equation*}
  \frac{1}{c_{\KM}} \trv(x) \leq r_{\mathbf{e}}(F_{\KM}(x)) \leq c_{\KM} \trv(x)
\end{equation*}
for all $x\in \KM$, with $\mathbf{e}$ being the edge such that $\mathbf{e}\cap
\partial \KM\neq \emptyset$.
It follows 
from \eqref{eq:td-d} and \eqref{eq:analytic-u-e-all} that, for all $\alpha\in \N^2_0$,
\begin{equation*}
\norm{\tilde{r}_{\mathbf{v}}^{\alpham-1/2-s+\varepsilon} \partial_{\tx}^\alpha \tilde{u} }_{L^2(F_{\KM}^{-1}(\KM))}
\leq 
\tC_{\varepsilon} \tgamma^{\alpham+1}\alpham^{\alpham}.
\end{equation*}
\paragraph{\textbf{Case $\KM\in \cT^\M_{\mathrm{int}}$.}} If $\KM\in
\cT^\M_{\mathrm{int}}$, then $\tu|_{F_{\KM}^{-1}(\KM)}$ is analytic. 

Since the macro triangulation $\cT^\M$ is fixed,
all constants depending on $\KM$ can be taken uniformly over the macro triangulation.
We have obtained that for all $\KM\in \cT^\M$, the restriction of $\tu$ to
$F^{-1}_{\KM}(\KM)$ satisfies the hypotheses of Lemmas \ref{lemma:R-corner-patch} to \ref{lemma:R-ce-patch}. 

Restricting
$\beta \in (1/2-s, 1-s)$ in \eqref{eq:u-v_triangle_proof} and using Lemmas
\ref{lemma:R-corner-patch} to \ref{lemma:R-ce-patch} gives therefore
\begin{equation*}
  \inf_{v\in W_q^L}\| u - v \|_{\tH^s(\Omega)} \leq C\left( \exp(-b_2 q)  + \exp(-b_1 L) \right).
\end{equation*}
Choosing $q\simeq L$, $V_N\coloneqq W^L_q$, and remarking that $\dim(W^L_q)\simeq q^2 L^2$ concludes the proof.
 \qed
\end{proof}
\section{Conclusions}
\label{sec:Concl}
We proved \emph{exponential rates of approximation for a class of 
  $hp$-Finite Element approximations}
of the Dirichlet problem for the integral fractional Laplacian
in a bounded, polygonal domain $\Omega \subset \mathbb{R}^2$, 
with analytic source term $f$,
based on anisotropic, geometric boundary-refined meshes.
The \emph{realization of corresponding $hp$-FE algorithms} will incur significant
issues of \emph{numerical quadrature for stable numerical evaluation of the bilinear form}
$a(\cdot,\cdot)$ in \eqref{eq:weakform} on pairs of large aspect ratio rectangles 
in the geometric boundary mesh patches shown in Fig.~\ref{fig:patches}.
While being in princple known (see, e.g., \cite{SS97_317} for a related discussion in 
$hp$ Galerkin boundary element methods on polyhedral domains), 
the corresponding consistency
analysis for the form $a(\cdot,\cdot)$ in \eqref{eq:weakform}
in the space $\widetilde{H}^s(\Omega)$
in \eqref{eq:Htilde} will be the topic of a forthcoming work.

Here, we analyzed only the 
convergence rate of the $hp$-Galerkin discretization \eqref{eq:GalV}
based on the subspaces $W^L_q$ in \eqref{eq:S^q_0} with geometric,
boundary-refined triangulations $\Tg$.
Similar techniques, based again on the anisotropic, weighted high-order Sobolev
regularity
of the solution $u$ of~\eqref{eq:FracLap} 
proved in \cite{FMMS21_983}, allow to infer
optimal algebraic rates of convergence 
$O(h^{q+1-s}) = O(N^{-(q+1-s)/2})$ 
in the 
$\tH^s(\Omega)$-norm
for continuous, piecewise polynomial Lagrangian Finite Elements of 
order $q\geq 1$.
To this end, however, the geometric boundary-refined partitions $\Tg$
in the design of the spaces  $W^L_q$ in \eqref{eq:S^q_0}  
must be replaced by \emph{boundary-refined graded partitions}.
Details shall be reported elsewhere.
\appendix
\normalsize
\section{Polynomial approximation operators on the reference element}
\label{sec:approx-reference-element}
The following two lemmas are consequences of 
\cite[Lemma 3.1, 3.2]{banjai-melenk-schwab19-RD}.
%
\begin{lemma}[approximation on triangles]
\label{lemma:hat_Pi_infty}
Let $\widehat T$ be the reference triangle. 
Then, for every $q \in {\mathbb N}$, 
there exists a linear operator 
$\widehat \Pi^\triangle_q:C^0(\overline{\widehat T})\rightarrow {\bbP}_q$ 
with the following properties: 
\begin{enumerate}
\item 
\label{item:lemma:hat_Pi_infty--1}
For each edge $e$ of $\widehat T$, 
$(\PiT u)|_e$ coincides with the 
Gauss-Lobatto interpolant $i_q (u|_e)$ of degree $q$ on the edge $e$.
\item (projection property)
\label{item:lemma:hat_Pi_infty-0}
$\PiT v = v$ for all $v\in {\bbP}_q$. 
%


\item 
\label{item:lemma:hat_Pi_infty-iii}
Let $u \in C^\infty(\widehat T)$ satisfy, 
for some $C_u$, $\gamma > 0$ 
$$ 
\forall n \in {\mathbb N}_0 \colon\;\;
\|\nabla^n u\|_{L^\infty(\widehat T)} 
\leq 
C_u \gamma^n
(n+1)^n . 
$$
  Then, there exist $C, b>0$ such that for all $q\geq 1$
  \begin{equation*}
\|u - \PiT u\|_{W^{1,\infty}(\widehat T)} \leq C C_u e^{-bq}.
  \end{equation*}
\end{enumerate}
\end{lemma}
%
%
\begin{lemma}[approximation on quadrilaterals]
\label{lemma:hat_Pi_1_infty}
Let $\widehat S$ be the reference square. 
For each $q \in {\mathbb N}$,
the tensor-product Gauss-Lobatto interpolation operator 
$\PiGL: C^0(\overline{\widehat S}) \rightarrow {\bbQ}_q$ 
satisfies the following: 
\begin{enumerate}
\item 
\label{item:lemma:hat_Pi_1_infty-0}
For each edge 
$e\subset \partial\widehat S$, 
$(\PiGL u)|_e$ coincides with the 
univariate Gauss-Lobatto interpolant $i_q (u|_e)$ on $e$.

\item (projection property)
\label{item:lemma:hat_Pi_1_infty--1}
$\PiGL v = v$ for all $v \in {\bbQ}_q$. 
%
\item 
\label{item:lemma:hat_Pi_1_infty-ii}
Let $u \in C^\infty(\widehat S)$ satisfy 
for some $C_u$, $\gamma > 0$, 
and all
$(n,m) \in {\mathbb N}_0^2 $
\begin{equation}
\label{eq:lemma:hat_Pi_1_infty-reg}
\|\partial_\xh^m \partial_\yh^n u\|_{L^\infty(\widehat S)} 
\leq 
C_u \gamma^{n+m}
(n+1)^n (m+1)^m.
\end{equation}
Then, there exist $C, b>0$ such that for all $q\geq 1$
 \begin{equation*}
\|u - \PiGL u\|_{W^{1,\infty}(\widehat S)}  \leq C C_u e^{-bq}.
 \end{equation*}
\end{enumerate}
\end{lemma}
\section{Estimates of norms with cutoff function}
\label{sec:norm-cutoff}

We introduce two technical lemmas.
\begin{lemma}
  \label{lemma:cutoff-equiv}
  Let $\cutoff$ be
  defined as in \eqref{eq:cutoff-prop} and let $\beta\in [0,1)$.
  Then, there exists $C>0$ such that with $\Omega^L_0$ defined in \eqref{eq:SLzero},
  it holds, for all $w\in H^1_{\beta}(\Omega)$ and all $L\in\mathbb{N}$, that
  \begin{equation*}
    \| \cutoff w\|_{H^1_\beta(\Omega)} \leq C \| w\|_{H^1_\beta (\Omega\setminus \Omega^L_0 )}.
  \end{equation*}
\end{lemma}
\begin{proof}
  By definition of $\cutoff$ we have for all $L\geq 1$
  \begin{equation*}
    \| \cutoff \|_{L^\infty(\Omega)} =1, \qquad \| \nabla \cutoff \|_{L^\infty(\Omega)} \simeq \sigma^{-L}.
  \end{equation*}
   In addition, there exists $c>0$ such that for all $L\geq 1$
  \begin{equation*}
    \supp(\nabla\cutoff) \subset S_{c\sigma^L} \setminus \Omega^L_0, 
          \qquad \supp(1-\cutoff) \subset S_{c\sigma^L}, \qquad \supp(\cutoff) \subset \Omega\setminus \Omega_{0}^L.
  \end{equation*}
  Hence, for all $L\geq 1$,
  \begin{align*}
    &\| r^{\beta-1}\cutoff w\|^2_{L^2(\Omega)} + 
    \| r^{\beta}\nabla(\cutoff w)\|^2_{L^2(\Omega)}\\
    &\quad \leq
    \| r^{\beta-1} w\|^2_{L^2(\Omega\setminus \Omega_{0}^L)} + 
      \|\nabla\cutoff\|^2_{L^\infty(\Omega)}
      \| r^{\beta} w\|^2_{L^2(S_{c\sigma^L}\setminus \Omega_{0}^L)}
      +
      \| r^{\beta} \nabla w\|^2_{L^2(\Omega\setminus \Omega_{0}^L)}
    \\ & \quad \lesssim
         \|  w\|^2_{H^1_\beta(\Omega\setminus \Omega_{0}^L)} + 
      \sigma^{-2L}\| r^{\beta} w\|^2_{L^2(S_{c\sigma^L}\setminus \Omega_{0}^L)}
    \\ & \quad \lesssim
         \|  w\|^2_{H^1_\beta(\Omega\setminus \Omega_{0}^L)} + 
      c^2\| r^{\beta-1} w\|^2_{L^2(S_{c\sigma^L}\setminus \Omega_{0}^L)},
  \end{align*}
with constants hidden in $\lesssim$ independent of $L$.
  \qed
\end{proof}
\begin{lemma}
  \label{lemma:1-cutoff}
  Let $\cutoff$ be
  defined as in \eqref{eq:cutoff-prop} and let $\beta\in [0,1)$.
  Then, there exist $C, c>0$ independent of $L$ such that, 
  for all $w\in H^1_{\beta}(\Omega)$ and all $L\in\mathbb{N}$,
  \begin{equation*}
    \| (1-\cutoff) w\|_{H^1_\beta(\Omega)} \leq C \| w\|_{H^1_\beta(S_{c\sigma^L})}.
  \end{equation*}
\end{lemma}
\begin{proof}
  The proof proceeds along the same lines as the proof of the previous lemma. We
  have
  \begin{equation*}
    \| r^{\beta-1} (1-\cutoff) w\|_{L^2(\Omega)} \leq 
\| r^{\beta-1}  w\|_{L^2(S_{c\sigma^L})},
  \end{equation*}
  where $c$ is defined as in the preceding proof such that 
  $r(x)\leq c\sigma^L$ holds for all $x\in \supp(1-\cutoff)$.
  Similarly, we obtain
  \begin{equation*}
    \| r^{\beta}\nabla((1-\cutoff )w)\|^2_{L^2(\Omega)}\lesssim 
\| r^{\beta}\nabla w\|^2_{L^2(S_{c\sigma^L})}
+ c^2 \| r^{\beta-1}w \|^2_{L^2(S_{c\sigma^L})},
  \end{equation*}
  which finishes the proof.\qed
\end{proof}
\bibliographystyle{plain}
\bibliography{biblio,bibliography}
\end{document}

%% file: preamble.tex
\let\epsilon\varepsilon
\newcommand{\tH}{\widetilde{H}}
\newcommand{\tu}{\widetilde{u}}
\newcommand{\tPi}{\widetilde{\Pi}}
\newcommand{\hPi}{\widehat{\Pi}}

\newcommand{\N}{\mathbb{N}}
\newcommand{\R}{\mathbb{R}}

\newcommand{\lay}{\mathcal{L}^L}

\newcommand{\tlayC}{\widetilde{\mathcal{L}}^{\Co, L}}
\newcommand{\tlayE}{\widetilde{\mathcal{L}}^{\Edg, L}}

\newcommand{\tlayCE}{\widetilde{\mathcal{L}}^{\CE, L}}

\newcommand{\tTCEint}{\tT^{\CE, L}_{\mathrm{int}}}
\newcommand{\tTCint}{\tT^{\Co, L}_{\mathrm{int}}}
\newcommand{\tSEint}{\tS^{\Edg, L}_{\mathrm{int}}}
\newcommand{\cutoff}{g^L}

\newcommand{\tS}{\widetilde{S}}

\newcommand{\Dpar}{D_{x_\parallel}}
\newcommand{\Dperp}{D_{x_\perp}}
\newcommand{\Ttrivial}{\check{\calT}^{\mathrm{trivial}}}
\newcommand{\Tgeneric}{\check{\calT}^{\bullet,L}_{\mathrm{geo},\sigma}}
\newcommand{\Tcor}{\check{\calT}^{\Co,L}_{\mathrm{geo},\sigma}}
\newcommand{\Tedg}{\check{\calT}^{\Edg,L}_{{\mathrm{geo}},\sigma}}
\newcommand{\Tce}{\check{\calT}^{\CE,L}_{{\mathrm{geo}},\sigma}}

\newcommand{\tD}{\widetilde{D}}
\newcommand{\hD}{\widehat{D}}

\newcommand{\hpartial}{\widehat{\partial}}
\newcommand{\hu}{\widehat{u}}
\newcommand{\hS}{\widehat{S}}
\newcommand{\hT}{\widehat{T}}
\newcommand{\tT}{\widetilde{T}}
\newcommand{\cT}{\mathcal{T}}
\newcommand{\trv}{\widetilde{r}_\mathbf{v}}
\newcommand{\tre}{\widetilde{r}_\mathbf{e}}
\newcommand{\alpham}{|\alpha|}

\newcommand{\tC}{\widetilde{C}}
\newcommand{\tgamma}{\widetilde{\gamma}}

\newcommand{\Tg}{\mathcal{T}^{L}_{\mathrm{geo},\sigma}}

\newcommand{\PiT}{\widehat \Pi^\triangle_q}

\newcommand{\PiGL}{\widehat \Pi^{\Box}_q}

\newcommand{\M}{\mathcal{M}}

\newcommand{\xh}{{\widehat x}}
\newcommand{\yh}{{\widehat y}}
\newcommand{\tx}{{\widetilde x}}
\newcommand{\ty}{{\widetilde y}}

\newcommand{\Kt}{{\widetilde K}}
\newcommand{\ut}{{\widetilde u}}

\newcommand{\Kh}{{\widehat  K}}

\newcommand{\Edg}{{\sf E}}

\newcommand{\Co}{{\sf V}}
\newcommand{\CE}{{\sf VE}}

\newcommand{\skp}[1]{\left< #1 \right>}
\newcommand{\norm}[1]{\left\| #1 \right\|}

\newcommand{\supp}{\operatorname*{supp}}

\newcommand{\eremk}{\hbox{}\hfill\rule{0.8ex}{0.8ex}}
\numberwithin{equation}{section}
\newcommand{\abs}[1]{\lvert#1\rvert}

\DeclareMathOperator{\dist}{dist}

\newcommand{\calT}{{\mathcal T}}

\newcommand{\bbN}{{\mathbb N}}

\newcommand{\bbQ}{{\mathbb Q}}
\newcommand{\bbP}{{\mathbb P}}
\newcommand{\bbR}{{\mathbb R}}

\newcommand{\tbe}{\widetilde{\mathbf{e}}}
\newcommand{\tbv}{\widetilde{\mathbf{v}}}
\newcommand{\bO}{{\mathbf O}}
\newcommand{\bP}{{\mathbf P}}

\newcommand{\VLq}{W^L_q}
\newcommand{\PiLq}{\Pi^L_q}
\newcommand{\KM}{K^{\M}}
\newcommand{\omegaeps}{\xi}
\newcommand{\omegac}{\omega_{\mathbf{v}}}
\newcommand{\omegacprime}{\omega_{\mathbf{v'}}}
\newcommand{\omegace}{\omega_{\mathbf{ve}}}
\newcommand{\omegaceprime}{\omega_{\mathbf{v'e'}}}
\newcommand{\omegae}{\omega_{\mathbf{e}}}
\newcommand{\omegaeprime}{\omega_{\mathbf{e'}}}

\newcommand{\dbeta}{\partial_x^\beta}
\newcommand{\dalpha}{\partial_x^\alpha}

\newcommand{\epar}{\mathbf{e}_\parallel}
\newcommand{\ppar}{p_{\parallel}}

\newcommand{\pperp}{p_{\perp}}